\documentclass[11pt,a4paper]{amsart}

\usepackage{amsmath,amssymb}
\usepackage{amsthm}
\usepackage{bm}
\usepackage{graphicx}
\newtheorem{definition}{Definition}[section]
  \newtheorem{theorem}[definition]{Theorem}
  \newtheorem{lemma}[definition]{Lemma}
  \newtheorem{corollary}[definition]{Corollary}
  \newtheorem{proposition}[definition]{Proposition}

\title{An estimation of Hempel distance by using\\ reeb graph}
\author{Ayako Ido}
\date{}
\thanks{Department of Mathematics, Nara Women's University, Nara, 630-8506, Japan. E-mail address: eaa.ido@cc.nara-wu.ac.jp}

\begin{document}
\maketitle

\begin{abstract}
Let $P, Q$ be Heegaard surfaces of a closed orientable 3-manifold. 
In this paper, we introduce a method for giving an upper bound of (Hempel) distance of $P$ by using the Reeb graph derived from a certain horizontal arc in the ambient space $[0,1]\times[0,1]$ of the Rubinstein-Scharlemann graphic derived from $P$ and $Q$. This is a refinement of a part of Johnson's arguments used for determining stable genera required for flipping high distance Heegaard splittings. 
 
\end{abstract}

\section{Introduction}	
Hempel \cite{He} introduced the concept of \textit{distance} of a Heegaard splitting, and it is shown by many authors that it well represents various complexities of 3-manifolds. For example, Scharlemann-Tomova shows that high distance Heegaard splittings are \lq\lq{rigid}\rq\rq. More precisely:

\medskip
\noindent
\textbf{Theorem}
\textit{(Corollary 4.5 of \cite{ST})
If a compact orientable 3-manifold $M$ has a genus $g$ Heegaard
surface $P$ with $d(P)>2g$, then}
\begin{itemize}
\item \textit{$P$ is a minimal genus Heegaard surface of $M$};

\item \textit{any other Heegaard surface of the same genus is isotopic to $P$}.
\end{itemize}
\textit{Moreover, any Heegaard surface $Q$ of $M$ with $2g(Q){\leq}d(P)$ is isotopic to a stabilization or boundary stabilization of $P$}.

\medskip

The above result is proved by using \textit{Rubinstein-Scharlemann graphic} (or \textit{graphic} for short). Graphic is introduced by Rubinstein-Scharlemann for studying Reidemeister-Singer distance of two strongly irreducible Heegaard splittings.
In \cite{KS}, Kobayashi-Saeki show that graphics for 3-manifolds can be regarded as the images of the discriminant sets of stable maps from the 3-manifolds into the plane $[0,1]\times[0,1]$, and as an application, they give an example (Corollary 5.7 of \cite{KS}) of a pair of Heegaard splittings such that a common stabilization of them can be observed as an arc in the ambient space $[0,1]\times[0,1]$ of the graphic. 
This approach is formulated in general setting by Johnson \cite{J1}, to give an estimation of the stable genera from above. 
He further developed the idea, and succeeded to determine stable genera required for flipping high distance Heegaard splittings \cite{J2}. (We note that this result is first proved by Hass, Thompson and Thurston \cite{HTT}.) 
One of the tools used in \cite{J2} is \textit{horizontal arcs} disjoint from \textit{mostly above regions} and \textit{mostly below regions} (for the definitions, see Section~5) in the ambient space of the graphic. 
By using such arcs, Johnson gives an estimation of distances of Heegaard splittings, which implies an alternative proof of the above result of Scharlemann-Tomova's.

In this paper, we give a more detailed treatment of such horizontal arcs, which can possibly give a better estimation of the distance.
In fact, given two strongly irreducible Heegaard splittings, we show that there exists a horizontal arc in the ambient space of the graphic derived from them, which is disjoint from $R_{X}\cup{R_{x}}$ and $R_{Y}\cup{R_{y}}$ (for the definitions, see Section~4). We show that there exists a subinterval of the horizontal arc whose interior is contained in unlabelled regions and adjacent to an $A$-region and a $B$-region. Then we give the definition of \textit{Reeb graph} $G$ derived from the horizontal arc. We consider the subgraph $G^{*}$ of $G$ corresponding to the above subinterval. Then we consider the subset $G^{*}_{e}$ consisting of edges corresponding to essential simple closed curves on $Q$. In Section~7, we introduce a method of assigning a positive integer to each edge of $G^{*}_{e}$. Then we have the following: 

\medskip
\noindent
\textbf{Theorem \ref{main}}
\textit{
Let $P, Q$ and $G^{*}_{e}$ be as above. Let $n$ be the minimum of the integers assigned to the edges adjacent to $\partial_{+}G^{*}_{e}$. 
Then the distance $d(P)$ is at most $n + 1$.}

\section{Preliminaries}

\noindent
\textbf{2.1. Heegaard splittings.}
A genus $g(\geq{1})$ \textit{handlebody} $H$ is the boundary sum of 
$g$ copies of a solid torus. 
Note that $H$ is homeomorphic to the closure of a regular neighborhood 
of some finite graph $\Sigma$ in $\mathbb{R}^{3}$. 
The image $\Sigma$ of the graph is called a \textit{spine} of $H$. 
By a technical reason, throughout this paper, we suppose that each vertex of spines of genus $g(>1)$ handlebodies is of valency three (for a detailed discussion see \cite{KS}, Sect.2). 
Let $M$ be a closed orientable 3-manifold. 
We say that $M=A\cup_{P}B$ is a (genus $g$) Heegaard splitting of $M$ 
if $A, B$ are genus $g$ handlebodies in $M$ 
such that $M=A \cup B$ and $A\cap B=\partial{A}=\partial{B}=P$. 
Then $P$ is called a (genus $g$) \textit{Heegaard surface} of $M$. 
A disk $D$ properly embedded in a handlebody $H$ is called a 
\textit{meridian disk} of $H$ if $\partial{D}$ is 
an essential simple closed curve in $\partial{H}$. 
A Heegaard splitting $M=A\cup_{P}B$ is \textit{stabilized}, 
if there are meridian disks $D_{A}, D_{B}$ of $A, B$ respectively 
such that $\partial{D}_{A}$ and $\partial{D}_{B}$ 
intersects transversely in a single point. 
We note that a genus $g$ Heegaard splitting $A\cup_{P}B$ is stabilized 
if and only if there exists a genus $g-1$ Heegaard splitting 
$A'\cup_{P'}B'$ such that $A\cup_{P}B$ is obtained from 
$A'\cup_{P'}B'$ by adding a \lq\lq trivial\rq\rq \ handle.
Then we say that $A\cup_{P}B$ is \textit{obtained from $A'\cup_{P'}B'$ 
by a stabilization}. 
We say that $A''\cup_{P''}B''$ is a stabilization of $A\cup_{P}B$, if $A''\cup_{P''}B''$ is obtained from $A\cup_{P}B$ by a finite number of stabilizations. 
If there are meridian disks $D_{A}, D_{B}$ in $A, B$ respectively 
so that $\partial{D}_{A}=\partial{D}_{B}$, 
$A\cup_{P}B$ is said to be \textit{reducible}. 
If there are meridian disks $D_{A}, D_{B}$ in $A, B$ respectively 
so that $\partial{D}_{A}, \partial{D}_{B}$ are disjoint on $P$, 
$A\cup_{P}B$ is said to be \textit{weakly reducible}. It is easy to see that if a Heegaard splitting $M=A\cup_{P}B$ is 
reducible, it is weakly reducible. If $A\cup_{P}B$ is not weakly reducible, it is said to be \textit{strongly irreducible}.

\noindent
\textbf{2.2. Curve complexes.} Let $S$ be a closed connected orientable surface $S$ of genus at least two, and $\mathcal{C}(S)$ the 1-skeleton of 
Harvey's complex of essential simple closed curves on $S$ (see \cite{Ha}), that is, $\mathcal{C}(S)$ denotes the graph whose 0-simplices are isotopy classes of essential simple closed curves and whose 1-simplices connect distinct 0-simplices with disjoint representatives. We remark that $\mathcal{C}(S)$ is connected. 
Let $x, y$ be 0-simplices of $\mathcal{C}(S)$. Then we define the distance between $x$ and $y$, denoted by $d_{S}(x,y)$, as the minimal of such $d$ that there is a path in $\mathcal{C}(S)$ with $d$ 1-simplices joining $x$ and $y$. Let $X, Y$ be subsets of the 0-simplices of $\mathcal{C}(S)$. Then we define
\vspace{-1.5mm}
\begin{center}
$d_{S}(X, Y)=\min\{d_{S}(x, y)\mid x\in{X}, y\in{Y}\}$.
\end{center}
\vspace{-1.5mm}
Suppose that $S$ is the boundary of a handlebody $V$. 
Then $\mathcal{M}(V)$ denotes the subset of $\mathcal{C}(S)$ consisting of the 0-simplices with representatives bounding meridian disks of $V$. For a genus $g(\geq 2)$ Heegaard splitting $A\cup_{P}B$, its Hempel distance, denoted by $d(P)$, is defined to be $d_{P}(\mathcal{M}(A),\mathcal{M}(B))$.

\section{Rubinstein-Scharlemann graphic}

Let $M, N$ be smooth manifolds. Then $C^{\infty}(M,N)$ denotes the space of the smooth maps of $M$ into $N$ endowed with the Whitney $C^{\infty}$ topology (see \cite{GG}, \cite{Hi}). 
Let $\varphi, \varphi' : M \rightarrow N$ be elements of $C^{\infty}(M,N)$. We say that $\varphi$ is \textit{equivalent} to $\varphi'$ if there exist diffeomorphisms $H : M \rightarrow M$ and $h : N \rightarrow N$ such that $\varphi'\circ H=h\circ \varphi$. The map $\varphi$ is said to be \textit{stable} if there is an open neighborhood $U_{\varphi}$ of $\varphi$ in $C^{\infty}(M,N)$ such that each $\varphi'$ in $U_{\varphi}$ is equivalent to $\varphi$. 

Let $M$ be a smooth closed orientable 3-manifold. 
A \textit{sweep-out} is a smooth map $f : M \rightarrow I$ such that for each $x \in (0, 1)$, the level set $f^{-1}(x)$ is a closed surface, and $f^{-1}(0)$ (resp. $f^{-1}(1)$) is a connected, finite graph such that each vertex has valency three. Each of $f^{-1}(0)$ and $f^{-1}(1)$ is called a \textit{spine} of the sweep-out. It is easy to see that each level surface of $f$ is a Heegaard surface of $M$ and the spines of the sweep-outs are spines of the two handlebodies in the Heegaard splitting. 
Conversely, given a Heegaard splitting $A \cup_{P} B$ of $M$, it is easy to see that there is a sweep-out $f$ of $M$ such that each level surface of $f$ is isotopic to $P$, $f^{-1}(0)$ is a spine of $A$, and $f^{-1}(1)$ is a spine of $B$.

Given two sweep-outs, $f$ and $g$ of $M$, we consider their product $f\times{g}$ (that is, $(f\times{g})(x) = (f(x),g(x))$), which is a smooth map from $M$ to $I{\times}I$. Kobayashi-Saeki \cite{KS} has shown that by arbitrarily small deformations of $f$ and $g$, we can suppose that $f \times g$ is a stable map on the complement of the four spines. 
At each point in the complement of the spines, the differential of the map $f \times g$ is a linear map from $\mathbb{R}^{3}$ to $\mathbb{R}^{2}$. This map have a one dimensional kernel for a generic point in $M$. The \textit{discriminant} set for $f \times g$ is the set of points where the differential has a higher dimensional kernel. Mather's classification of stable maps \cite{JM} implies that: at each point of the discriminant set, the dimension of the kernel of the differential is two, and: the discriminant set is a one dimensional smooth submanifold in the complement of the spines in $M$. Moreover the discriminant set consists of all the points where a level surface of $f$ is tangent to a level surface of $g$ (here, we note that the tangent point is either a \lq\lq{center}\rq\rq \ or \lq\lq{saddle}\rq\rq). 

Let $f$, $g$ be as above with $f \times g$ stable. 
The image of the discriminant set is a graph in $I \times I$, which is  called the \textit{Rubinstein-Scharlemann graphic}. It is known that the Rubinstein-Scharlemann graphic is a finite 1-complex $\Gamma$ with each vertex having valency four or two. Each valency four vertex is called a \textit{crossing vertex}, and each valency two vertex is called a \textit{birth-death vertex}. There are valency one or two vertices of the graphic on the boundary of $I \times I$. Each component of the complement of $\Gamma$ in $I \times I$ is called a \textit{region}. At each point of a region, the corresponding level surfaces of $f$ and $g$ are disjoint or intersect transversely. 
The stable map $f \times g$ is \textit{generic} if each arc $\{s \} \times I$ or $I \times \{ t \}$ contains at most one vertex of the graphic. By Proposition 6.14 of \cite{KS}, by arbitrarily small deformation of $f$ and $g$, we may suppose that $f \times g$ is generic.

\section{Labelling regions of the graphic}

Let $f$ and $g$ be sweep-outs obtained from Heegaard splittings $A{\cup_{P}}B$, $X{\cup_{Q}}Y$, respectively with $f{\times}g$ stable. For each $s \in ( 0, 1 )$, we put that 
$P_{s}=f^{-1}(s)$ $(P_{0}=\Sigma_{A}, P_{1}=\Sigma_{B})$, $A_{s}=f^{-1}([0,s])$ and $B_{s}=f^{-1}([s,1])$.
Similarly, for $t\in(0,1)$, we put that $Q_{t}=g^{-1}(t)$ $(Q_{0}=\Sigma_{X}, Q_{1}=\Sigma_{Y})$, $X_{t}=g^{-1}([0,t])$ and $Y_{t}=g^{-1}([t,1])$. 
Let $(s,t)$ be a point in a region of the graphic. Then either $P_{s}{\cap}Q_{t}=\phi$, or 
$P_{s}$ and $Q_{t}$ intersect transversely in a collection 
$C=\{c_{1}, \dots, c_{n}\}$ of simple closed curves.

\begin{definition} 
Let $C=\{c_{1}, \dots, c_{n}\}$ be as above. Then $C_{P}$ denotes the subset of $C$ 
consisting of the elements which are essential on $P_{s}$.
Furthermore the subset $C_{A}$ of $C_{P}$ is defined by:
\vspace{-1.5mm}
\begin{center}
$C_{A}=\{c \mid c$ bounds a disk $D$ in $Q_{t}\setminus{C_{P}}$ such that $N(\partial{D},D)\subset A_{s}$\}, 
\end{center}
\vspace{-1.5mm}
where $N(\partial{D},D)$ is a regular neighborhood of $\partial{D}$ in $D$.
Analogously $C_{B} \subset C_{P}$ and $C_{X}, C_{Y}{\subset}C_{Q}$ are defined.
\end{definition}

Then we note that the following facts are known.

\begin{lemma}(Corollary 4.4 of \cite{RS})
If there exists a region such that both $C_{A}$ and $C_{B}$ (resp. $C_{X}$ and $C_{Y}$) are non-empty, 
then $A \cup_{P} B$ (resp. $X\cup_{Q}Y$) is weakly reducible.
\label{j}
\end{lemma}

\begin{lemma}(Lemma 4.5 of \cite{RS})
Suppose that $C_{P}$ and $C_{Q}$ are empty, and there exists a meridian disk $D$ in $A_{s}$ which intersects $Q_{t}$ only in inessential simple closed curves. Moreover, suppose that there is an essential simple closed curve $l$ on $Q_{t}$ such that $l \subset A_{s}$. Then either $A \cup_{P} B$ is weakly reducible or $M$ is the 3-sphere $S^{3}$. 
The statement obtained by substituting $(A,P,Q)$ in the above with $(B,P,Q)$, $(X, Q,P)$ or $(Y,Q,P)$ also hold. 
\label{a}
\end{lemma}

Now we introduce how to label each region with following the convention of \cite{RS}. 
If $C_{A}$ (resp. $C_{B}$, $C_{X}$, $C_{Y}$) is non-empty, 
the region is labelled $A$ (resp. $B, X, Y$). 
If $C_{P}$ and $C_{Q}$ are both empty and $A$ (resp. $B$) contains an 
essential curve of $Q$, then the region is labelled $b$ (resp. $a$). If $C_{P}$ and $C_{Q}$ are both empty and $X$ (resp. $Y$) contains an essential curve of $P$, 
then the region is labelled $y$ (resp. $x$). 
$R_{A}$ (resp. $R_{B}, R_{X}, R_{Y}, R_{a}, R_{b}, R_{x}, R_{y}$) denotes the closure of the union of the regions labelled $A$ (resp. $B, X, Y, a, b, x, y$). $R_{\phi}$ denotes the closure of the union of the unlabelled regions. Lemma \ref{j} shows that if there is a region with both labels $A$ and $B$, then the Heegaard splitting $A\cup_{P}B$ is weakly reducible. Moreover:

\begin{lemma}(Corollary 5.1 of \cite{RS})
If there exist two adjacent regions such that one is labelled $A$(resp. $X$) and the other is labelled $B$(resp. $Y$), then $A\cup_{P}B$ (resp. $X\cup_{Q}Y$) is weakly reducible.
\label{two adjacent regions}
\end{lemma}

The proof of the next lemma can be found in the paragraph preceding Proposition 5.9 of \cite{RS}.

\begin{lemma}
Suppose that $A\cup_{P}B$ and $X\cup_{Q}Y$ are strongly irreducible and $M\neq{S^{3}}$. Then each region adjacent to $\{0\}\times{I}$ (resp. $\{1\}\times{I}$, $I\times\{0\}$, $I\times\{1\}$) is labelled $A$ or $a$ (resp. $B$ or $b$, $X$ or $x$, $Y$ or $y$).\label{near edge}
\end{lemma}

\section{Spanning and splitting sweep-outs}

In this section, we introduce the idea in \cite{J2}, which is used to give a lower bound of the number of stabilizations required for flipping the given Heegaard splittings and give a refinement of the formulation. 
Let $P_{s}, A_{s}, B_{s}, Q_{t}, X_{t}, Y_{t}$ be as in Section 4 with $f{\times}g$ generic. 
Suppose that $( s,t )$ is a point in a region. 
We say that $P_{s}$ is \textit{mostly above} $Q_{t}$ if each component of $P_{s} \cap X_{t}$ is contained in a disk subset of $P_{s}$. 
$P_{s}$ is \textit{mostly below} $Q_{t}$ if each component of $P_{s} \cap Y_{t}$ is contained in a disk subset of $P_{s}$. 
Now $R_{P>Q}$ denotes the closure of the union of the regions where $P_{s}$ is mostly above $Q_{t}$ and $R_{P<Q}$ denotes the union of the regions where $P_{s}$ is mostly below $Q_{t}$. 

According to \cite{J2}, we say that $g$ \textit{splits} $f$ if there exists $t$ such that $(I \times \{t \}) \cap (R_{P>Q} \cup R_{P<Q}) = \phi$. We say that $g$ \textit{spans} $f$ if $g$ does not split $f$, i.e., for all $t$, we have $(I \times \{t \}) \cap (R_{P>Q} \cup R_{P<Q})\neq\phi$. 
For a proof of the next lemma, see the paragraph preceding Lemma 15 in \cite{J2}.

\begin{lemma}
Suppose $M$ is irreducible. 
If $Q$ is not isotopic to $P$ or a stabilization of $P$, then $g$ splits $f$. 
\end{lemma}

For each $t \in (0,1)$, the pre-image in $f\times{g}$ of the arc $I\times\{t\}$ is the level surface $Q_{t}$, and the restriction of $f$ to $Q_{t}$ is a function with critical points in the level.

\begin{lemma}(Lemma 21 of \cite{J2})
If $g$ splits $f$, there exists $t$ such that $I \times \{t \}$ is disjoint from $R_{P>Q} \cup R_{P<Q}$ and the restriction of $f$ to $Q_{t}$ is a Morse function such that for each regular value $s$, $P_{s} \cap Q_{t}$ contains a simple closed curve which is essential on $P_{s}$. \label{h}
\end{lemma}

In the rest of this section, we give a refinement of the above arguments.

\begin{lemma}
Suppose $P, Q$ are strongly irreducible and $(s,t)$ is in a region contained in ${R_{P>Q}}$. 
If there exists a component of $P_{s}{\cap}Q_{t}$ which is essential on $Q_{t}$, then the region containing $(s,t)$ is labelled $X$. 
\end{lemma}

\begin{proof}
Let $C_{Q}$ be as in Section 2. Since $(s,t){\in}R_{P>Q}$, each element of $C_{Q}$ is inessential on $P_{s}$. Let $c$ be an element of $C_{Q}$ which is innermost on $P_{s}$, and $D$($\subset{P_{s}}$) be the disk bounded by $c$. If $N(\partial{D},D)\subset{X_{t}}$, the region is labelled $X$. Assume, for a contradiction, that $N(\partial{D},D)\subset{Y_{t}}$. Let $P_{s}^{*}$ be the component of $P_{s}{\setminus}Q_{t}$ which contains a simple closed curve that is essential on $P_{s}$. Note that since $(s,t){\in}R_{P>Q}$, (1) each component of $P_{s}{\setminus}P_{s}^{*}$ is a disk, and (2) $P_{s}^{*}{\subset}Y_{t}$. By (1), we see that there is an ambient isotopy $\psi_{t}$ of $M$ such that $\psi_{1}(\Sigma_{A}){\subset}P_{s}^{*}$. Let $P_{*}$ be the boundary of a sufficiently small regular neighborhood of $\psi_{1}(\Sigma_{A})$ (hence, $P_{*}$ is isotopic to $P$). By (2), we see that $P_{*}\subset{Y_{t}}$. Then we can apply Lemma \ref{a} to $P_{*}, Q_{t}$ and $D$ to show that $Q_{t}$ is weakly reducible, a contradiction. (Here we note that in Lemma \ref{a}, the Heegaard surfaces $P_{s}, Q_{t}$ are level surfaces. However it is easy to see that the proof of Lemma 4.5 of \cite{RS} works for the Heegaard surfaces $P_{*}, Q_{t}$ and $D$.) This completes the proof. 
\end{proof}

By Lemma 5.3, we see that if $(s,t)$ is in a region contained in ${R_{P>Q}}$, then we have one of the following; (1) the region containing $(s, t)$ is labelled $X$ (this holds in case when there exists a component of $P_{s} \cap Q_{t}$ which is essential on $Q_{t}$), (2) the region containing $(s, t)$ is labelled $x$ (this holds in case when each component of $P_{s} \cap Q_{t}$ is inessential on $Q_{t}$). These show that $R_{P>Q} \subset R_{X}\cup R_{x}$. Analogously $R_{P<Q} \subset R_{Y}\cup R_{y}$.

We say that $g$ \textit{strongly splits} $f$ if there exists $t$ such that $I \times \{t \}$ is disjoint from $(R_{X}\cup R_{x})\cup({R_{Y} \cup R_{y})}$. The next proposition was suggested by Dr. Toshio Saito. 
Here we note that after submitting the first version of this paper, the author realized that a result of Tao Li (Lemma 3.2 of \cite{Li}) implies the proposition as a special case. However our proof has a different flavor from that of Li's, and we decided to leave our proof in this paper. 

\begin{proposition} Let $M, P, Q$ be as above. 
Suppose $P, Q$ are strongly irreducible. If $Q$ is not isotopic to $P$, then $g$ strongly splits $f$. \label{b}
\end{proposition}

\begin{proof}
Suppose that $g$ does not strongly splits $f$. Then there exist values $s_{-}, s_{+}, t$ such that $(s_{\pm}, t) \in R_{X} \cup R_{x}$ and $(s_{\mp}, t) \in R_{Y}\cup R_{y}$. We have the following cases.

\vspace{1.5mm}
\noindent
\textbf{Case 1.} $(s_{\pm}, t) \in R_{X}, (s_{\mp}, t) \in R_{Y}$. 

Without loss of generality, we may suppose that $(s_{-}, t) \in R_{X}$ and $(s_{+}, t) \in R_{Y}$. In this case, $P_{s_{-}} \cap Q_{t}$ contains a simple closed curve which is essential on $Q_{t}$ and bounds a disk in $X_{t}$ while $P_{s_{+}} \cap Q_{t}$ contains a simple closed curve which is essential on $Q_{t}$ and bounds a disk in $Y_{t}$. This shows that $Q$ is weakly reducible, a contradiction. 

\vspace{1.5mm} 
\noindent
\textbf{Case 2.} $(s_{\pm}, t) \in R_{x}, (s_{\mp}, t) \in R_{y}$.

Without loss of generality, we may suppose that $(s_{-}, t) \in R_{x}$ and $(s_{+}, t) \in R_{y}$. 
In this case, by an isotopy, we may suppose that $Q_{t}$ is contained in $f^{-1}([s_{-},s_{+}])$ ${\cong}P{\times}[s_{-},s_{+}]$. If $Q_{t}$ is incompressible in $P \times [s_{-}, s_{+}]$, $Q$ is isotopic to $P$ (Corollary 3.2 of \cite{W}), a contradiction. If $Q_{t}$ is compressible in $P \times [s_{-}, s_{+}]$ then there is a compression disk $D$ such that $D \subset P \times [s_{-}, s_{+}]$. $D$ is contained in $X_{t}$ or $Y_{t}$. By applying Lemma \ref{a} to $P_{s_{-}}, Q_{t}$ and $D$ (if $P_{s_{-}}$ and int$D$ are contained in  the same component of $M\setminus{Q_{t}}$) or $P_{s_{+}}, Q_{t}$ and $D$ (if $P_{s_{+}}$ and int$D$ are contained in  the same component of $M\setminus{Q_{t}}$), we see that $Q$ is weakly reducible, a contradiction.

\vspace{1.5mm}
\noindent 
\textbf{Case 3.} $(s_{\pm}, t) \in R_{x}, (s_{\mp}, t) \in R_{Y}$ (or $(s_{\pm}, t) \in R_{X}, (s_{\mp}, t) \in R_{y}$). 

Without loss of generality, we may suppose that $(s_{-}, t) \in R_{x}$ and $(s_{+}, t) \in R_{Y}$.
In this case, since $(s_{-},t)\in{R_{x}}$, each component of $P_{s_{-}}\cap{Q_{t}}$ is inessential on both $P_{s_{-}}$ and $Q_{t}$, and there is an essential simple closed curve $\ell$ in $P_{s_{-}}$ such that $\ell\subset{Y_{t}}$. Let $C^{+}$ be the collection of simple closed curve(s) consisting of $P_{s_{+}}\cap{Q_{t}}$, then $C^{+}_{P}$ denotes the subset of $C^{+}$ which are essential on $Q_{t}$. Since $(s_{+},t)\in{R_{Y}}$, there is a disk component, say $E$,  of $P_{s_{+}}\setminus{Q_{t}}$ such that $N(\partial{E},E)\subset{Y_{t}}$. Since $M$ admits a strongly irreducible Heegaard splitting, $M$ is irreducible. Hence there is an ambient isotopy $\psi_{t}$ $(0\leq{t}\leq{1})$ of $M$ realizing disk swaps between $E$ and $Q_{t}$ such that $\psi_{1}(E)$ is a meridian disk of $Y_{t}$. Here we note that each component of $\psi_{1}(P_{s_{-}})\cap{Q_{t}}$ is inessential on both $\psi_{1}(P_{s_{-}})$ and $Q_{t}$, and there is an essential simple closed curve $\ell'$ in $\psi_{1}(P_{s_{-}})$ such that $\ell'\subset{Y_{t}}$. By applying Lemma \ref{a} to $Q_{t}$, $\psi_{1}(P_{s_{-}})$, and $\psi_{1}(E)$, we see that $Q$ is weakly reducible, a contradiction. 
\end{proof}

\begin{figure}[htbp]
 \begin{center}
 \includegraphics[width=37mm]{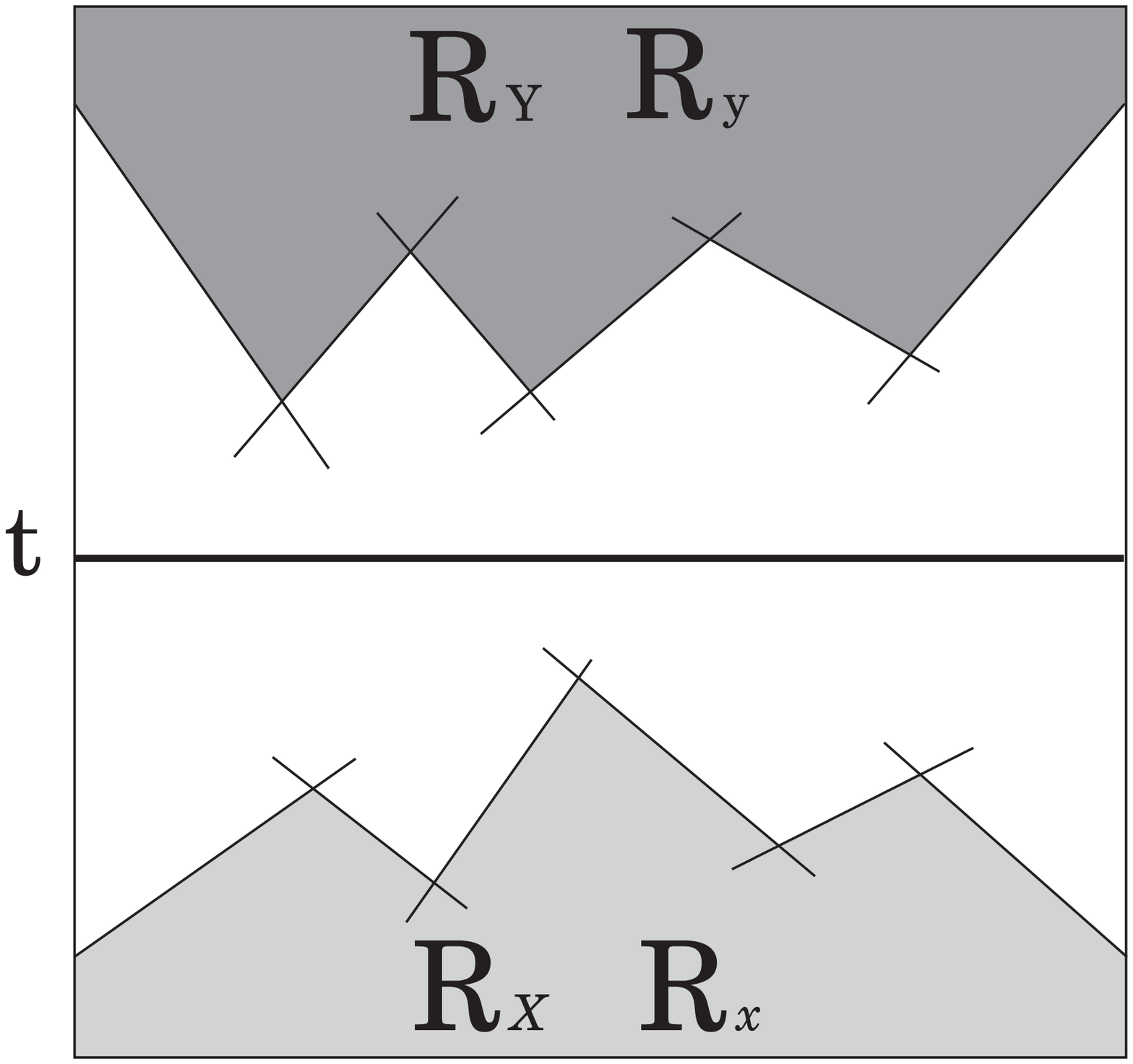}
 \end{center}
 \caption{}
\label{fig:three}
\end{figure}

We note that the arguments in the proof of Lemma \ref{h} work for the arc in Proposition \ref{b}. Hence we have: 

\begin{lemma}
If $g$ strongly splits $f$, there exists $t$ such that $I \times \{t \}$ is disjoint from $(R_{X}\cup R_{x})\cup(R_{Y}\cup R_{y})$ and the restriction of $f$ to $Q_{t}$ is a Morse function such that for each regular value $s$, $P_{s} \cap Q_{t}$ contains a simple closed curve which is essential on $P_{s}$.\label{strongly splits morse}
\end{lemma}

\begin{corollary}
Let $t$ be as in Lemma \ref{strongly splits morse}. 
There is a subarc $[s_{0}, s_{1}] \times \{t \} \subset I \times \{t \}$ such that:

\noindent 
\begin{itemize}
  \setlength{\parskip}{0cm} 
  \setlength{\itemsep}{0cm} 
 \item $(s_{0}, t) \in$ \{an edge of the graphic\}, 
 \item $(s_{1}, t) \in$ \{an edge of the graphic\}, and
 \item for any $s \in (s_{0}, s_{1})$, $(s, t) \in R_{\phi}$, and for any small $\epsilon > 0$, $((s_{0}-\epsilon), t) \in R_{A}$ and $((s_{1}+\epsilon), t) \in R_{B}$.
\end{itemize} \label{g}
\end{corollary}

\begin{proof}
By Lemma \ref{strongly splits morse}, $I \times \{t \}$ is disjoint from $(R_{X}\cup R_{x})\cup(R_{Y} \cup R_{y})$. By Lemma \ref{near edge}, a neighborhood of $(0,t)$ (resp. $(1,t)$) in $[0,1] \times \{t \}$ is contained in $R_{A}\cup R_{a}$ (resp. $R_{B}\cup R_{b}$). For an $s\in(0,1)$, if $(s,t)$ is contained in $R_{a}$ or $R_{b}$, then $(s, t)$ is contained in $R_{x}$ or $R_{y}$, a contradiction. Hence for a small $\epsilon>0$, $(\epsilon,t)$ (resp. $(1-\epsilon,t)$) is contained in $R_{A}$ (resp. $R_{B}$). Let $s_{1}=\sup\{s{\mid[0,s]\times\{t\}\in{R_{A}{\cup}R_{\phi}}}\}$ and $s_{0}=\sup\{s<s_{1}{\mid(s,t)\in{R_{A}}}\}$. Then by Lemma \ref{two adjacent regions}, $s_{0}{\neq}s_{1}$, and it is clear that the conclusion of Corollary 5.6 holds. 
\end{proof}

\section{The Reeb Graph}

Given a compact, orientable surface $F$, let $\varphi$ : $F \rightarrow \mathbb{R}$ be a smooth function such that $\varphi\mid_{int(F)}$ is a Morse function and each component of $\partial F$ is level. 
Define the equivalence relation $\sim$ on points on $F$ by $x \sim y$ whenever $x, y$ are in the same component of a level set of $\varphi$. The \textit{Reeb graph} corresponding to $\varphi$ is the quotient of $F$ by the relation $\sim$.
As suggested by the name, the Reeb graph $F'=F/\sim$ is a graph such that 
the edges of $F'$ come from annuli in $F$ fibered by level loops, and that the valency one vertices correspond to center singularities, and the valency three vertices correspond to saddle singularities. 

Let $f$ and $g$ be as in Section 5. Suppose that $Q$ is not isotopic to $P$, and we take $t$ as in Lemma \ref{strongly splits morse}. Let $G$ be the Reeb graph corresponding to $f\mid_{Q_{t}}$. 
There are two types of edges in $G$. If each point of an edge corresponds to an essential simple closed curve on $Q_{t}$, then the edge is called an \textit{essential edge}. 
If each point of the edge corresponds to an inessential simple closed curve on $Q_{t}$, then the edge is called an \textit{inessential edge}.

\begin{figure}[htb]
  \begin{minipage}[t]{.47\textwidth}
    \centering
    \includegraphics[width=36mm]{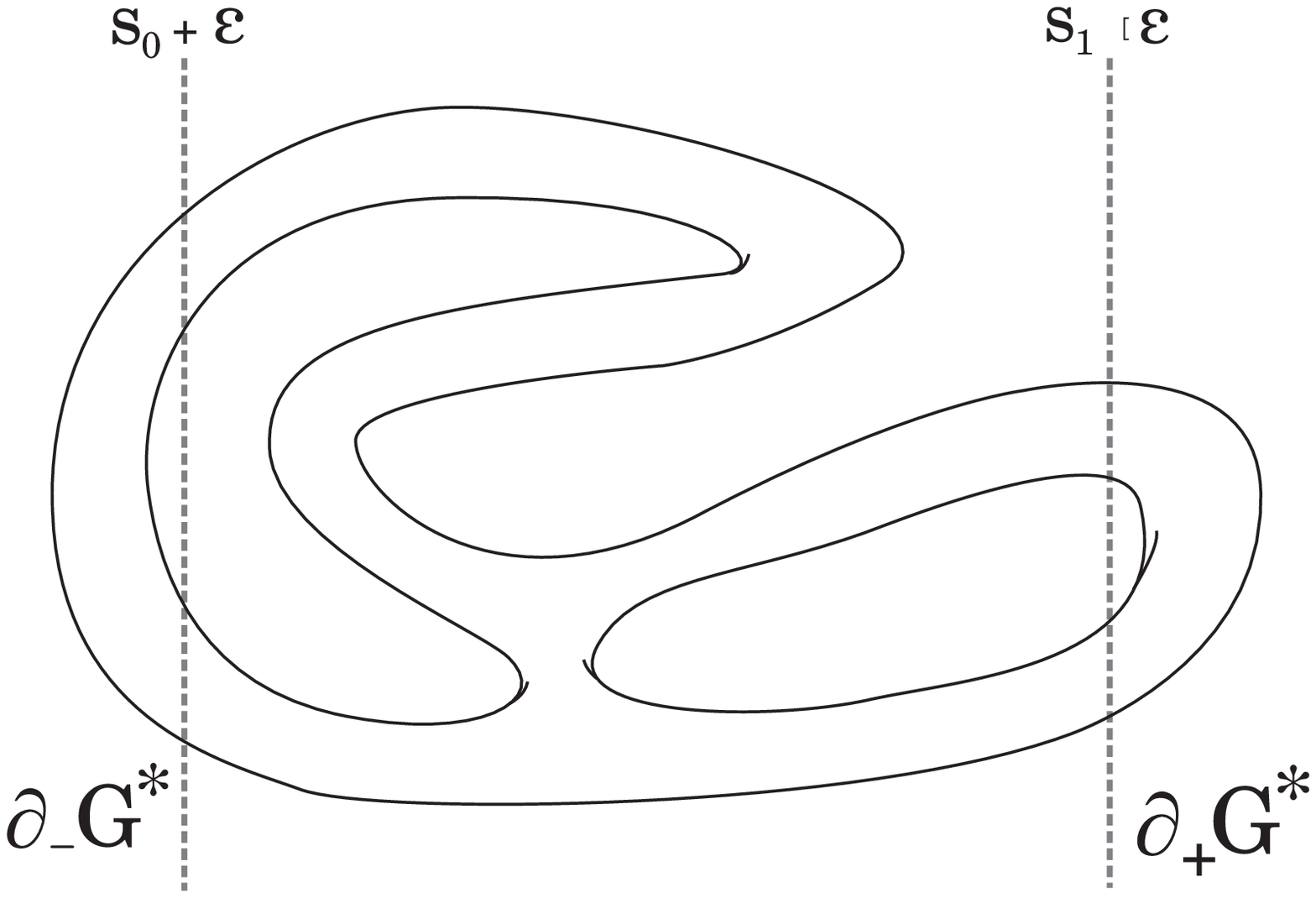}
    \caption{}
    \label{fig:one}
  \end{minipage}
  \hfill
  \begin{minipage}[t]{.47\textwidth}
    \centering
    \includegraphics[width=36mm]{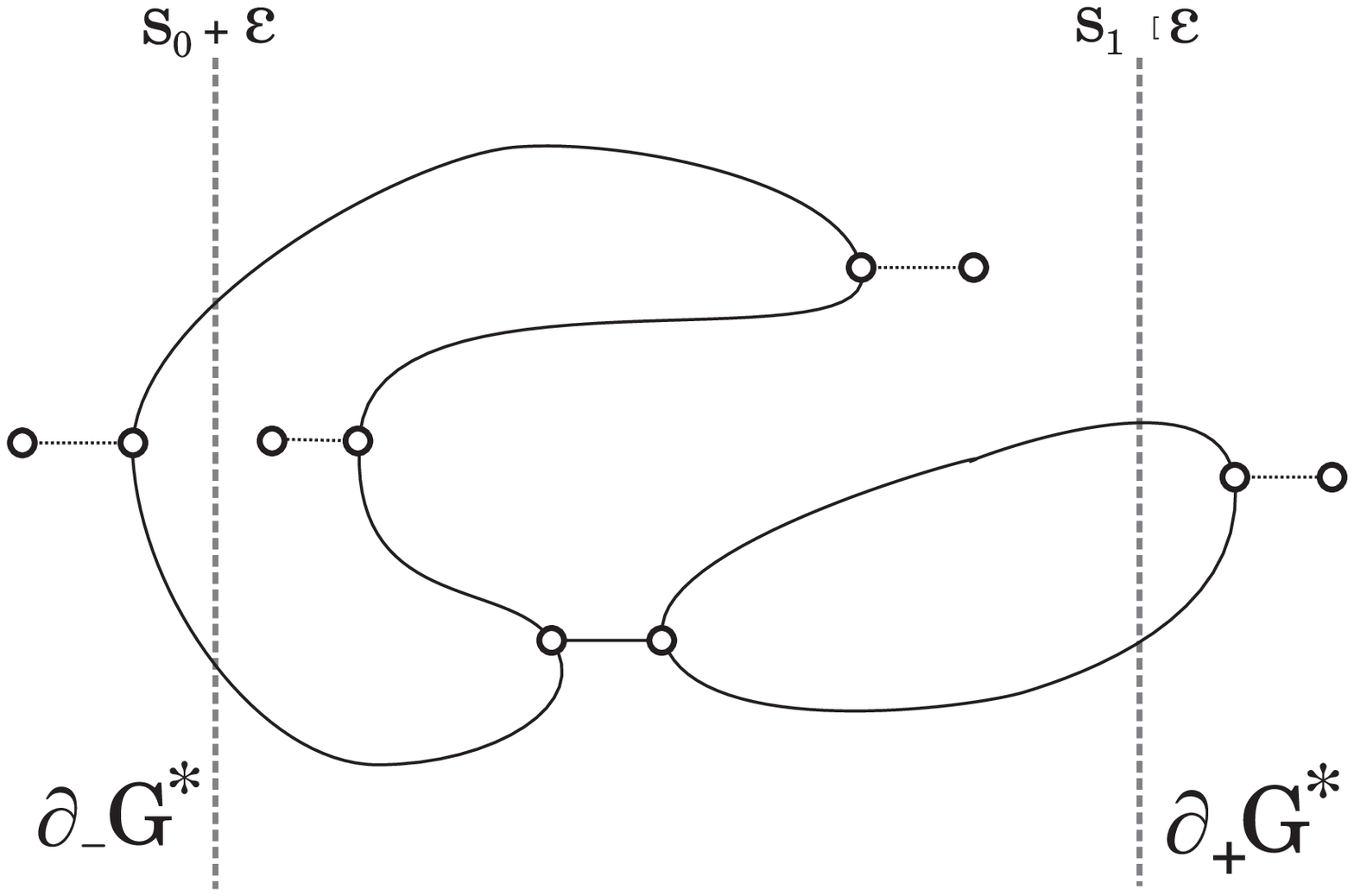}
    \caption{}
    \label{fig:two}
  \end{minipage}
\end{figure}

We continue with hypotheses of Section 5. Particularly, let $s_{0}, s_{1}, t$ be as in Corollary \ref{g}, hence, $[s_{0}, s_{1}] \times \{ t\}$ is an unlabelled interval in horizontal arc. For a small $\epsilon>0$, let $Q^{*}=f^{-1}(s_{0}+\epsilon,s_{1}-\epsilon)\cap Q_{t}$.
Then $G^{*}$ denotes the Reeb graph corresponding to $f\mid_{Q^{*}}:Q^{*}\rightarrow I$. We say that a vertex of $G^{*}$ corresponding to a component of $\partial Q^{*}$ is a \textit{$\partial$-vertex}. In particular, 
if a $\partial$-vertex corresponds to a component of $f^{-1}(s_{0} + \epsilon) \cap Q_{t}$ (resp. $f^{-1}(s_{1} - \epsilon) \cap Q_{t}$), then it is called a \textit{$\partial_{-}$-vertex} (resp. \textit{$\partial_{+}$-vertex}). The union of $\partial_{-}$-vertices (resp. $\partial_{+}$-vertices) is denoted by $\partial_{-}G^{*}$ (resp. $\partial_{+}G^{*}$). Let $f^{*}:G^{*}\rightarrow[s_{0}+\epsilon,s_{1}-\epsilon]$ be the function induced from $f\mid_{Q^{*}}:Q^{*}\rightarrow[s_{0}+\epsilon,s_{1}-\epsilon]$. Note that for each $s\in(s_{0}+\epsilon,s_{1}-\epsilon)$, $f^{*-1}(s) (\subset{G^{*}})$ consists of a finite number of points corresponding to the components of $P_{s}\cap{Q_{t}}$. Since $(s,t)$ is contained in an unlabelled region, there exists a component of $P_{s}{\cap}Q_{t}$ which is essential on both surfaces. This implies the next proposition. 

\begin{proposition}
Let $e$ be an inessential edge of $G^{*}$. For each $x \in e$, there exists an essential edge $e'$ of $G^{*}$ such that $f^{*}(x) \in f^{*} (e')$.
\end{proposition}

We consider local configurations of essential edges and inessential edges near a valency three vertex. 
At a valency three vertex, we may regard that an edge branches away two edges or that two edges are bound into one according to the parameter $s$ $(\in (s_{0}+\epsilon, s_{0}-\epsilon))$. 
We first consider the case of branching away (Figure 4). We take a point (e.g. $x_{1}, x_{2}$ and $x_{3}$) in each edge adjacent to the vertex as in Figure 4. Then $c_{x_{i}}$ denotes the simple closed curve on $Q_{t}$ corresponding to $x_{i}$. Each $c_{x_{i}}$ is either essential or inessential on $Q_{t}$. 
We naively have six cases up to reflection in horizontal line. 
But the two cases in Figure~4 do not occur, because it is easy to see that if $c_{x_{1}}$ and $c_{x_{2}}$ (resp. $c_{y_{2}}$ and $c_{y_{3}}$) in Figure~\ref{valency three branch} are inessential simple closed curves on $Q_{t}$, then $c_{x_{3}}$ (resp.  $c_{y_{1}}$) is also inessential.

\begin{figure}[htbp]
 \begin{center}
  \includegraphics[width=55mm]{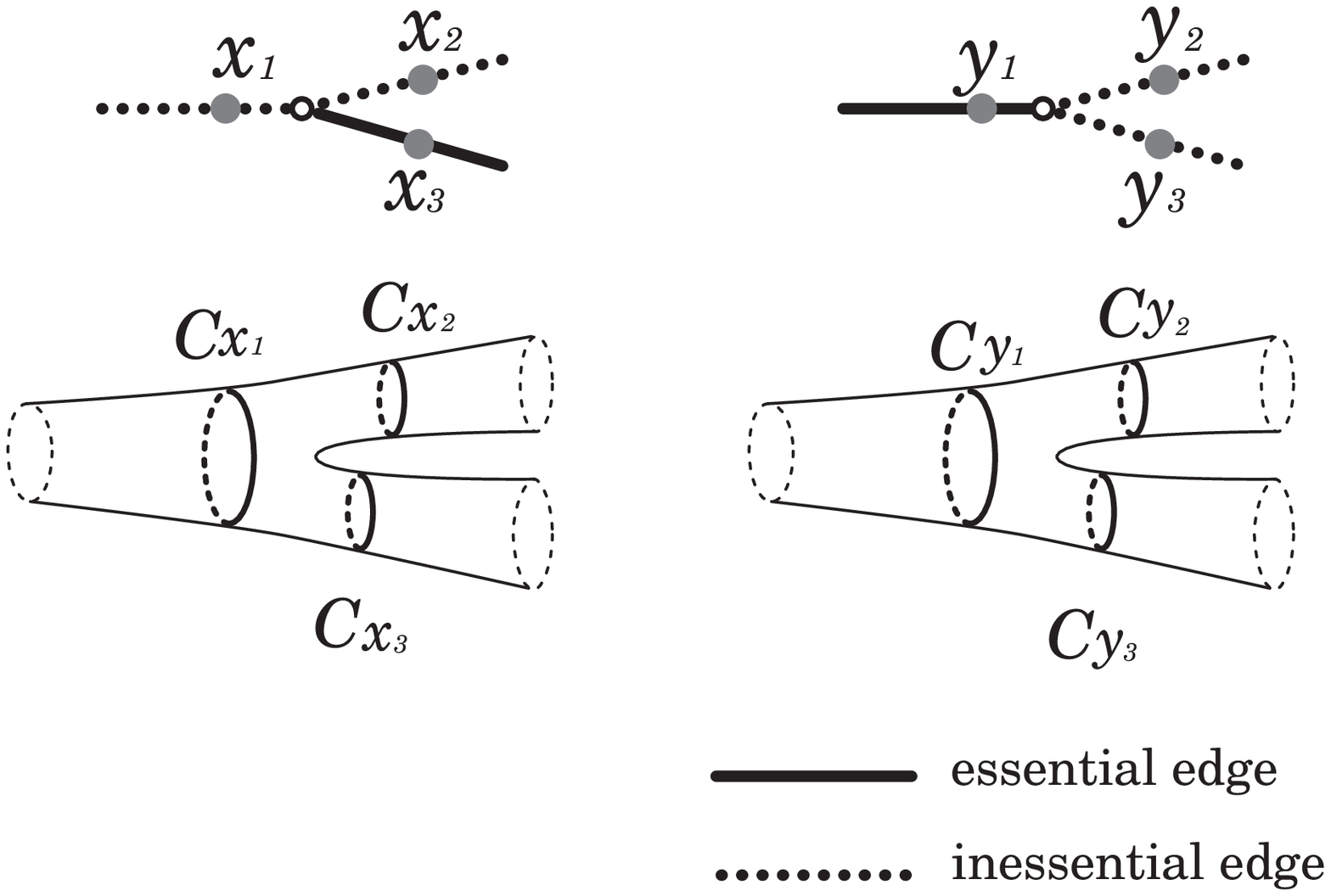}
 \end{center}
 \caption{}
 \label{valency three branch}
\end{figure}

Hence, the possible patterns of essential and inessential edges in neighborhoods of the vertices are shown in Figure~\ref{edge pattern} (1)--(4).

Type~(1) shows that an essential edge branches away two essential edges and type~(2) shows that an inessential edge branches away two essential edges. Type~(3) shows that an essential edge branches away an essential edge and an inessential edge and type~(4) shows that an inessential edge branches away two inessential edges. 

Then we consider the case of binding into one edge. It is clear that possible cases are obtained from type~(1)--(4) configurations by a horizontal reflection, which are shown in Figure~\ref{edge pattern} (5)--(8).

\vspace{3mm}

\begin{figure}[htbp]
 \begin{center}
  \includegraphics[width=50mm]{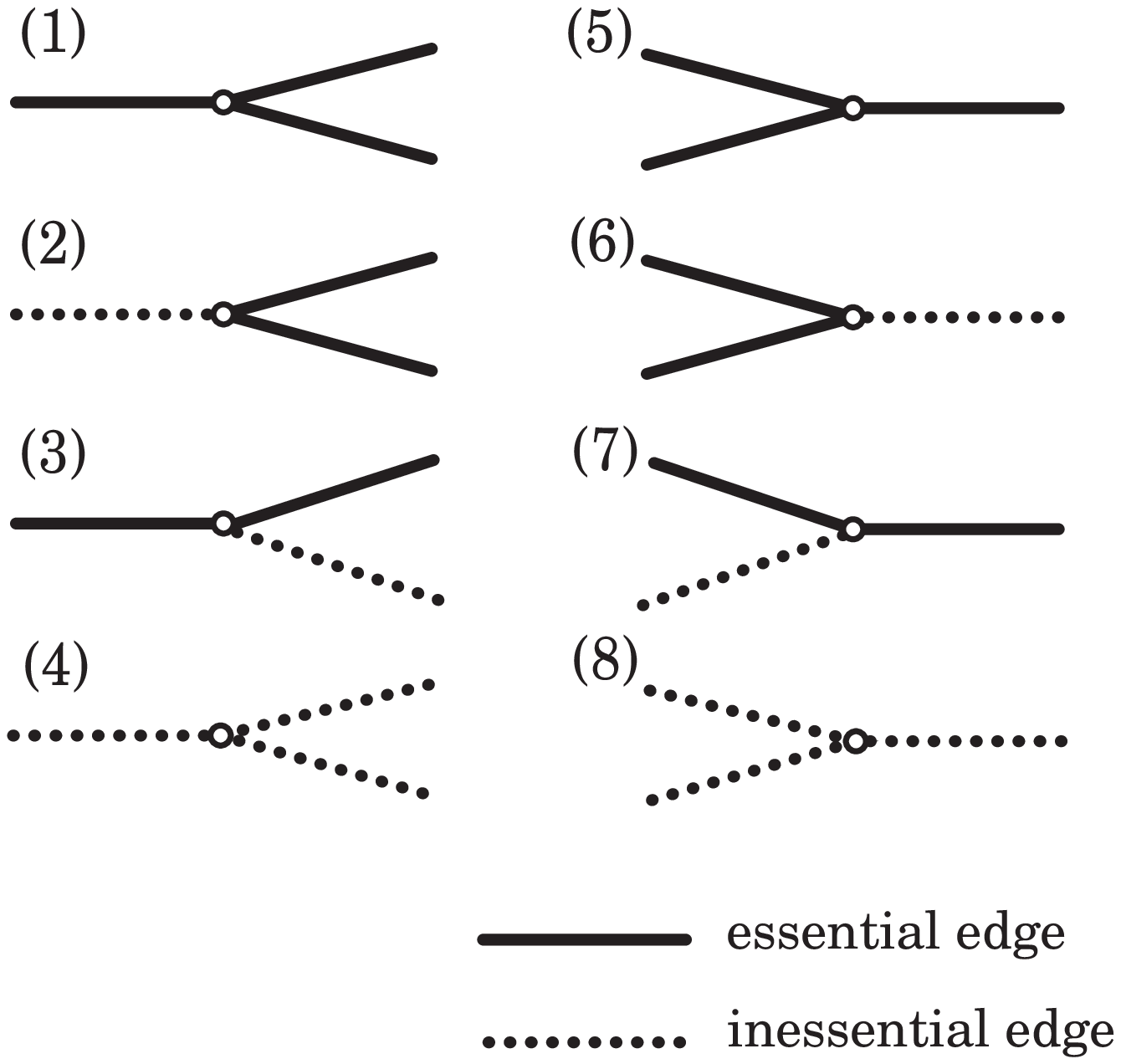}
 \end{center}
 \caption{}
 \label{edge pattern}
\end{figure}

\section{An estimation of Hempel distance}

\noindent
\textbf{Assigning positive integers to essential edges}

\vspace{1mm}
Let $G^{*}$, $\partial_{\pm}G^{*}$, $f^{*}$ be as in Section~6. 
Let $G^{*}_{e}$ be the subgraph of $G^{*}$ consisting of the essential edges of $G^{*}$. Then $\partial_{\pm}G^{*}_{e}$ denotes the vertices of $G^{*}_{e}$ corresponding to $\partial_{\pm}G^{*}$. We assign a positive integer to each edge of $G^{*}_{e}$ according to the following steps. Let $v_{1},\dots, v_{k}$ be the vertices of $G^{*}_{e}$ which are not $\partial$-vertices. 
We suppose that $v_{1},\dots, v_{k}$ are positioned in this order from the left, i.e., $f^{*}(v_{1}) < f^{*}(v_{2}) < \dots <f^{*}(v_{k})$. 

\vspace{5mm}

\begin{figure}[htbp]
\begin{center}
  \includegraphics[width=60mm]{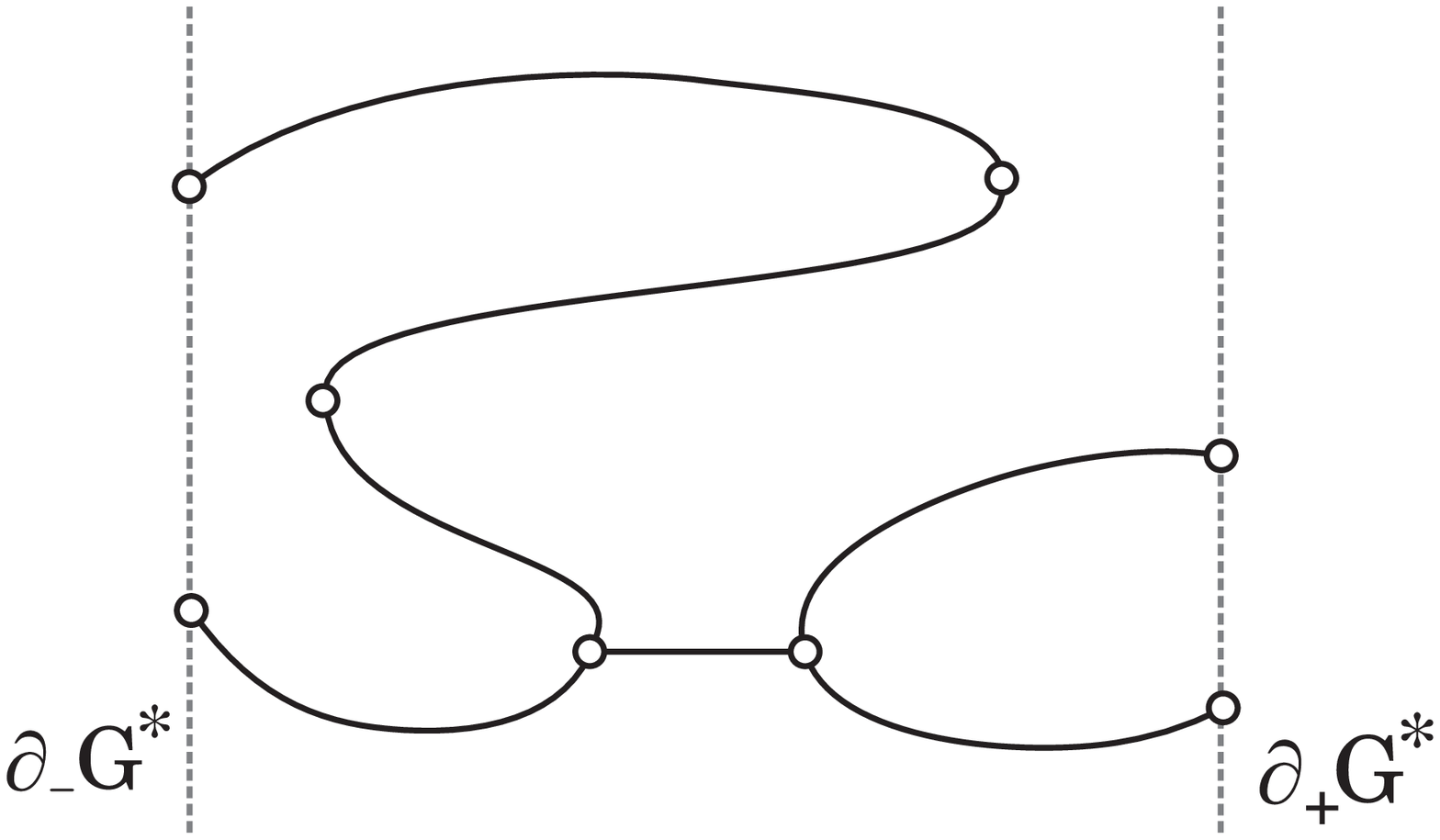}
 \end{center}
 \caption{}
 \label{fig:three}
\end{figure}

Now we define Steps~0, 1 and 2 inductively for assigning positive integers to the edges of $G^{*}_{e}$. 
\vspace{2mm}

\noindent
\textbf{Step 0.}
We assign $1$ to every edge adjacent to $\partial_{-} G^{*}_{e}$.

\begin{figure}[htbp]
\begin{center}
  \includegraphics[width=22mm]{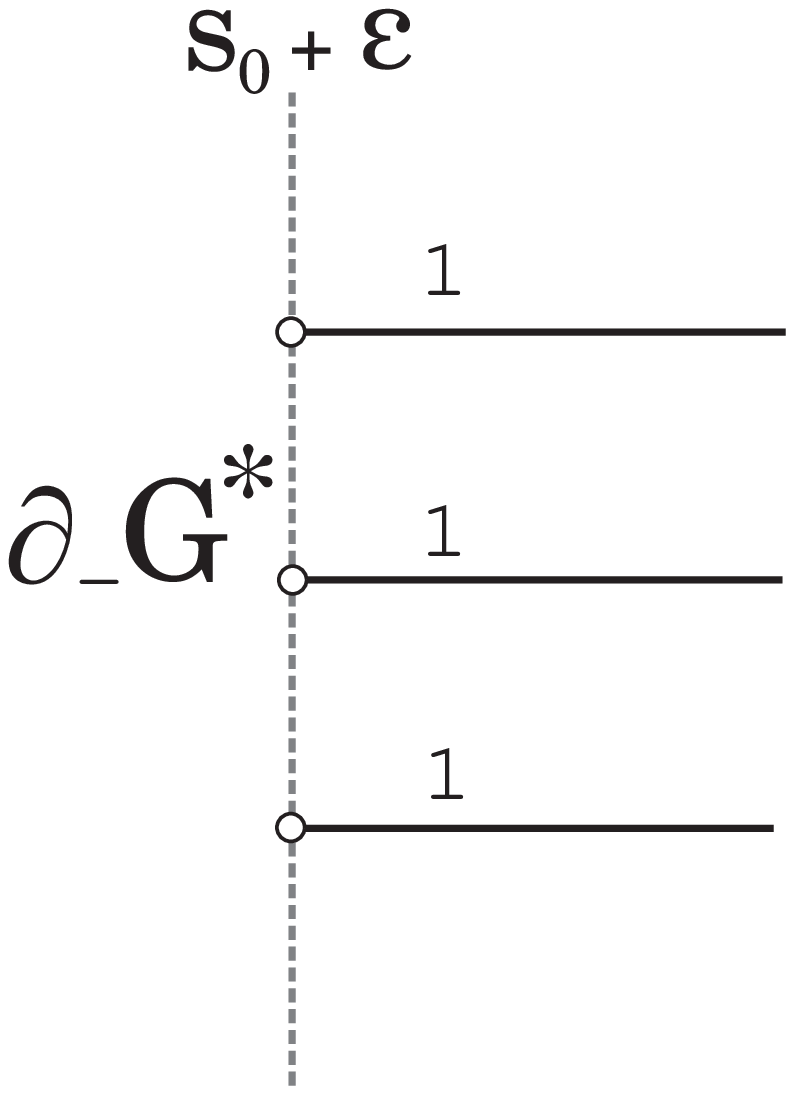}
 \end{center}
 \caption{}
 \label{fig:three}
\end{figure}

\vspace{2mm}

\noindent
\textbf{Step 1.}
Suppose that there is a valency two vertex $v_{i}$ adjacent to edges 
$e_{l}, e_{l'}$ such that $e_{l}$ has already been assigned and $e_{l'}$ has not been assigned yet. Then we assign the same integer as that of $e_{l}$ to $e_{l'}$. We apply this assignment as much as possible.

\begin{figure}[htbp]
\begin{center}
  \includegraphics[width=45mm]{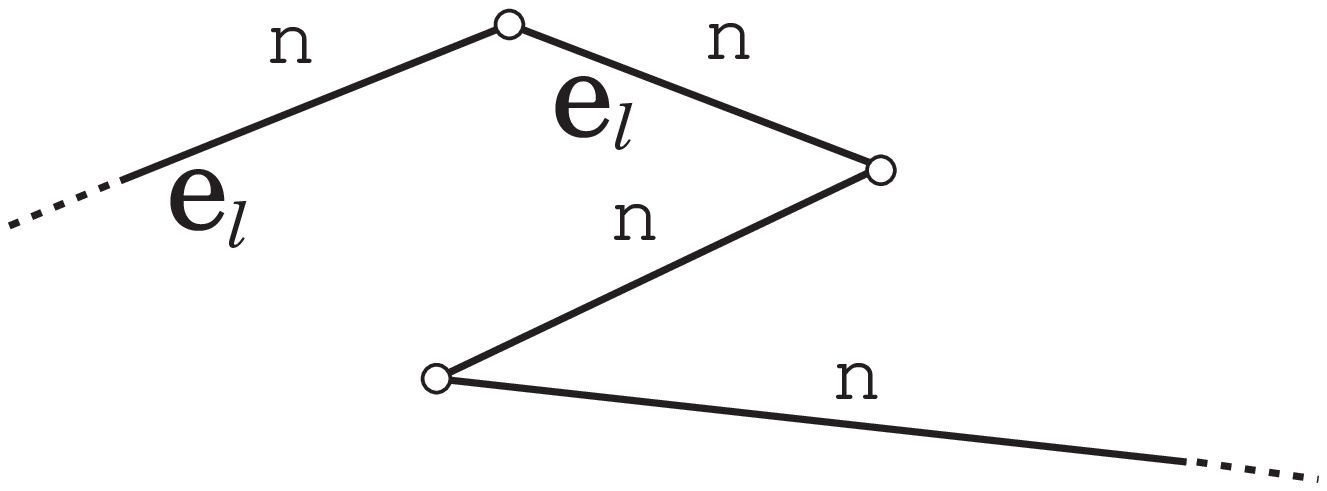}
 \end{center}
 \caption{}
 \label{fig:three}
\end{figure}

\vspace{2mm}

In our assigning process, we will repeat applications of Steps~1 and 2. Before describing Step~2, we will give a general condition that the assignments have in the process. 
Suppose we finish Step~1 in repeated applications of Steps~1 and 2. 
At this stage, either every edge of $G^{*}_{e}$ is assigned exactly one integer, or there is a unique vertex $v_{i}$ such that there is an unassigned edge adjacent to $v_{i}$, and that each edge of $G_{e}^{*}$ containing a point $x$ with $f^{*}(x)<f^{*}(v_{i})$ has already been assigned exactly one integer.
Then we suppose that the assigned integers satisfy the following conditions (*) and (**). (Note that the conditions are clearly satisfied after Steps~0 and 1.)

\medskip
\noindent
\textbf{(*)} For a small $\epsilon>0$, let $L_{i}$ be the set of the edges of $G^{*}_{e}$ each containing a point $x$ with $f^{*}(x)=f^{*}(v_{i})-\epsilon$. Then it satisfies one of the following conditions:

\medskip

\noindent
\textbf{(1)} All of the elements of $L_{i}$ are assigned with the same integer, say $n$.

\begin{figure}[htb]
  \begin{minipage}[t]{.47\textwidth}
    \centering
    \includegraphics[width=23mm]{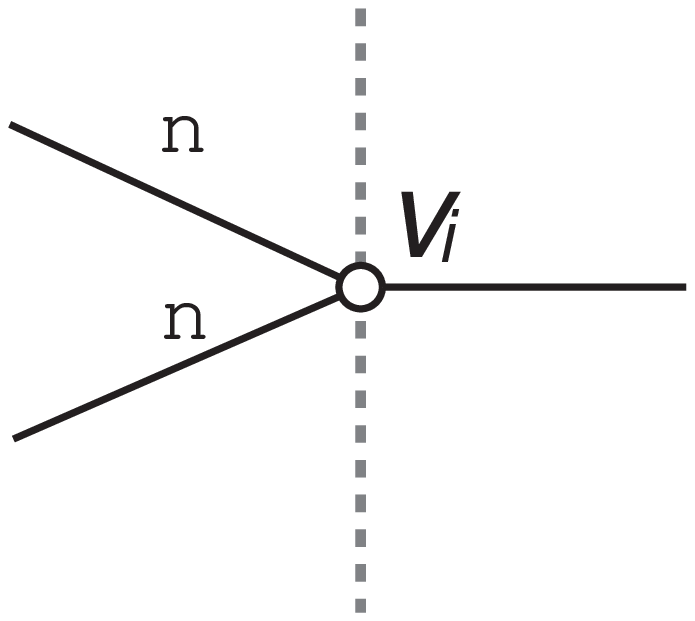}
    \caption{}
    \label{fig:one}
  \end{minipage}
  \hfill
  \begin{minipage}[t]{.47\textwidth}
    \centering
    \includegraphics[width=23mm]{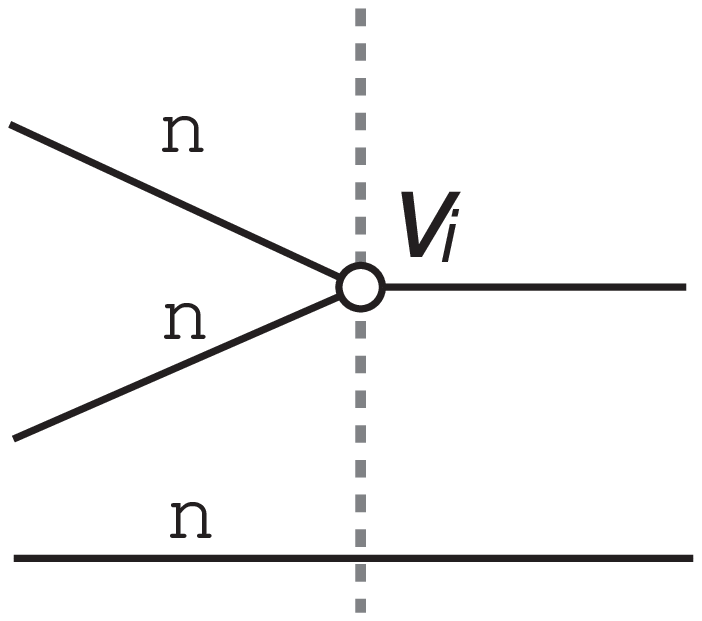}
    \caption{}
    \label{fig:two}
  \end{minipage}
\end{figure}

\noindent
\textbf{(2)} The set of the integers assigned to the elements of $L_{i}$ consists of consecutive two integers, say $n-1$ and $n$.

\begin{figure}[htb]
  \begin{minipage}[t]{.47\textwidth}
    \centering
    \includegraphics[width=23mm]{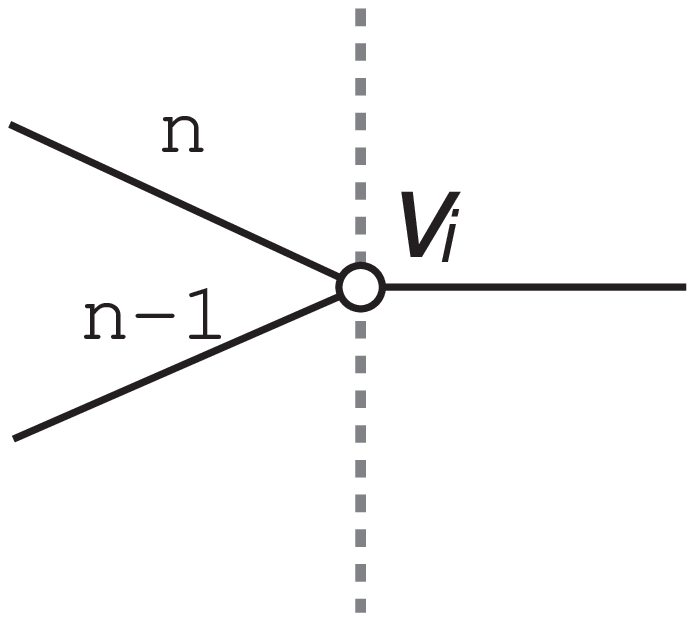}
    \caption{}
    \label{fig:one}
  \end{minipage}
  \hfill
  \begin{minipage}[t]{.47\textwidth}
    \centering
    \includegraphics[width=23mm]{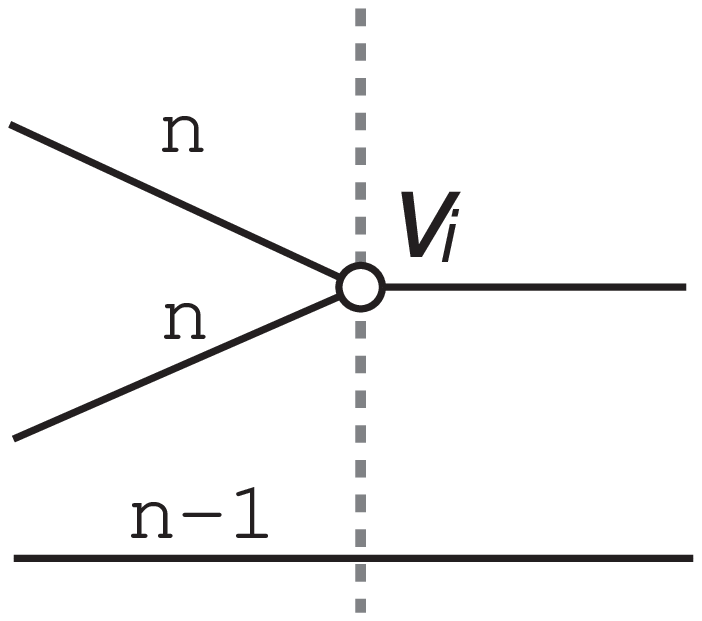}
    \caption{}
    \label{fig:two}
  \end{minipage}
\end{figure}

\medskip

\noindent
\textbf{(**)} Moreover, if there exists an assigned edge of $G_{e}^{*}$ containing a point $p_{+}$ with $f^{*}(p_{+})>f^{*}(v_{i})$, then each such edge is assigned $n-1$ or $n$ as above. In particular, for a point $p_{-}$ with $f^{*}(p_{-})>f^{*}(v_{i})-\epsilon$, if all of the assigned edges each containing a point $x$ with $f^{*}(x)=f^{*}(p_{-})$ are assigned the same integer $n'$ ($=n-1$ or $n$), then we have: 
\vspace{1.2mm}

\noindent
\textbf{(i)} all of the assigned edges containing a point $x_{+}$ with $f^{*}(x_{+})>f^{*}(p_{-})$ are assigned $n'$, and  

\noindent
\textbf{(ii)} if a point $x^{*}$ with $f^{*}(x^{*})>f^{*}(p_{-})$ is contained in an assigned edge, then there is a point $x'$ with $f^{*}(x')=f^{*}(p_{-})$ and a path $\mathcal{P}$ in $G_{e}^{*}$ such that $\mathcal{P}$ is contained in a union of edges assigned $n'$, and joins the points $x'$ and $x^{*}$. 
\medskip

\begin{figure}[htbp]
 \begin{center}
  \includegraphics[width=50mm]{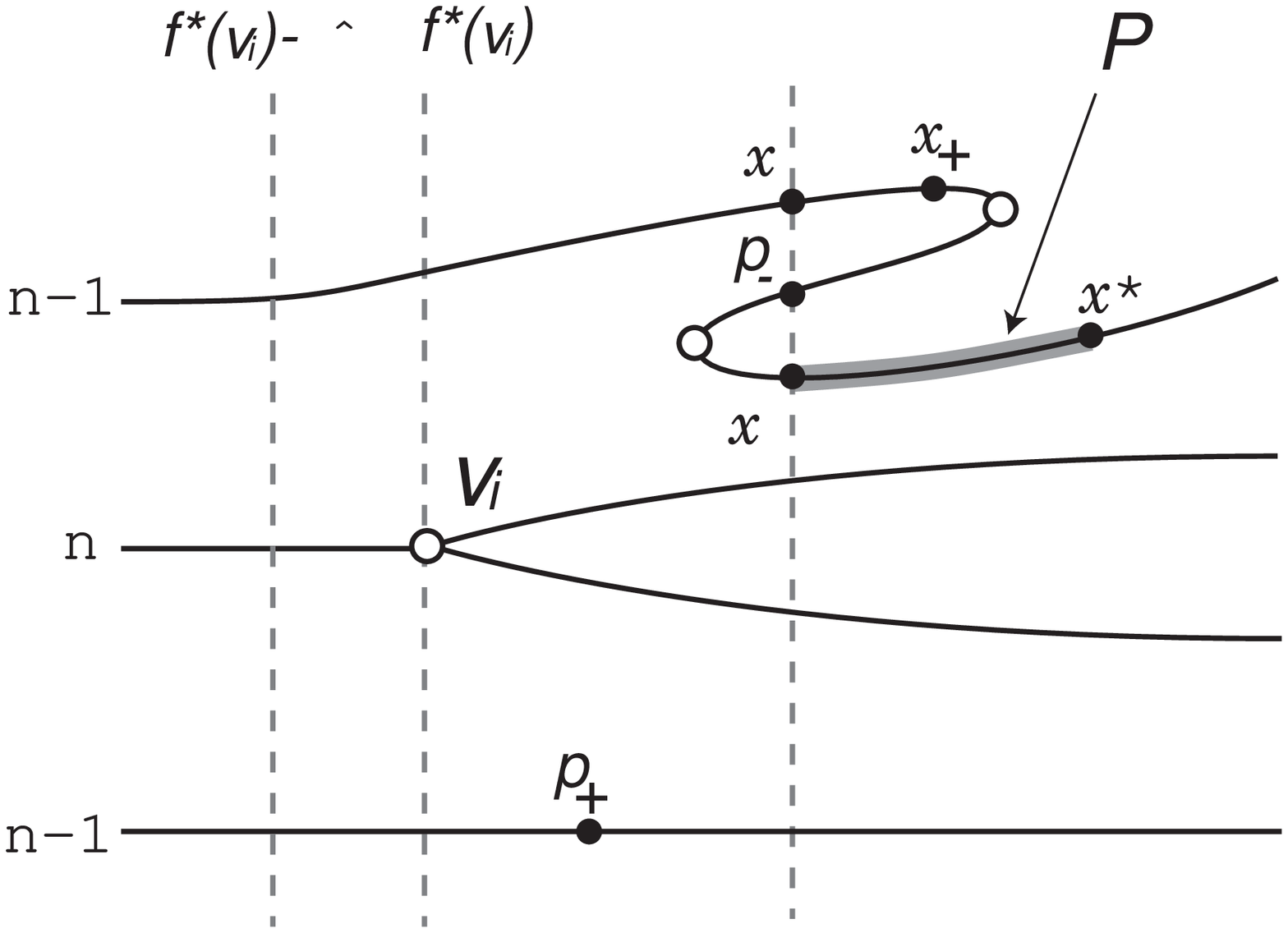}
 \end{center}
 \caption{}
 \label{}
\end{figure}

\noindent
\textbf{Step~2.} 

\medskip

\noindent
\textbf{1.} Suppose that the vertex $v_{i}$ satisfies the condition (1). 
Then we assign $n+1$ to the unassigned edge(s) adjacent to $v_{i}$.

\begin{figure}[htbp]
 \begin{center}
  \includegraphics[width=108mm]{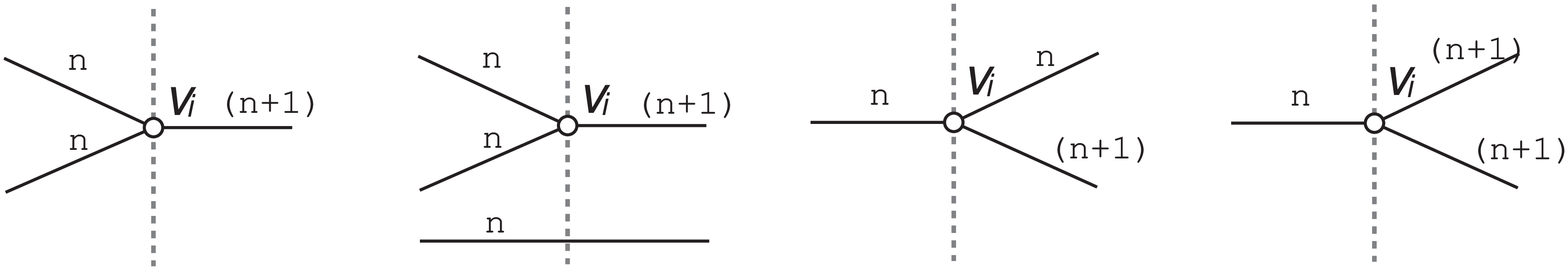}
 \end{center}
 \caption{}
 \label{}
\end{figure}

\noindent
\textbf{2.} Suppose that the vertex $v_{i}$ satisfies the condition (2). 
Then we assign $n$ to the unassigned edge(s) adjacent to $v_{i}$.

\begin{figure}[htbp]
 \begin{center}
  \includegraphics[width=108mm]{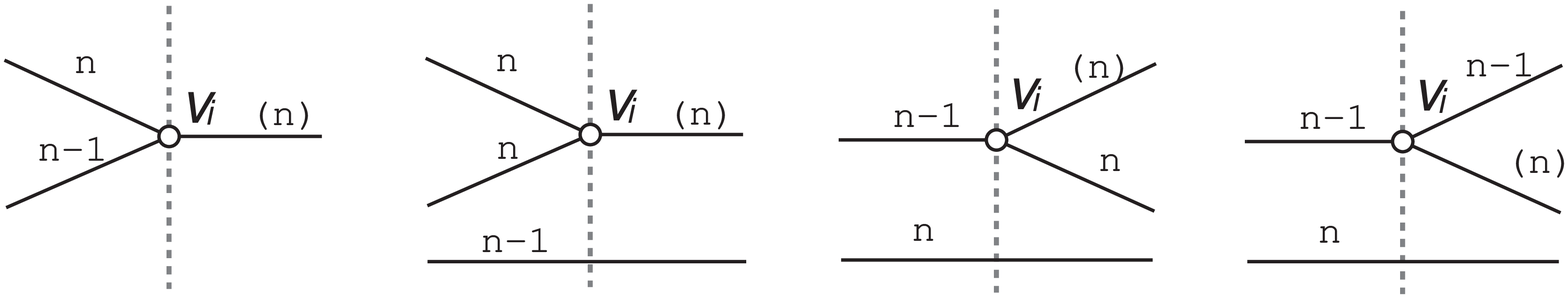}
 \end{center}
 \caption{}
 \label{}
\end{figure}

\medskip

After finishing Step~2, we apply Step~1. Then either all of the edges of $G^{*}_{e}$ are assigned integer(s), or there is a unique vertex $v_{j}$ such that there is an unassigned edge adjacent to $v_{j}$, and that each edge of $G^{*}_{e}$ that contains a point $y$ such that $f^{*}(y)<f^{*}(v_{j})$ has already been assigned. Here we note that there are no multiple assignments of integers for an edge and the above conditions (*) and (**) hold for the new assignment. That is:

\medskip

\begin{lemma}
At the stage, each edge of $G^{*}_{e}$ is assigned at most one integer, and we have the following. 

\medskip
\noindent
\textbf{(*)$'$} For a small $\epsilon>0$, let $L_{j}$ be the set of the edges of $G^{*}_{e}$ each containing a point $y$ with $f^{*}(y)=f^{*}(v_{j})-\epsilon$. Then the assigned integers satisfy one of the following conditions:
\vspace{1.5mm}

\noindent
\textbf{(1)$'$} All of the elements of $L_{j}$ are assigned with the same integer, say $m$.

\noindent
\textbf{(2)$'$} The set of the integers assigned to the elements of $L_{j}$ consists of consecutive two integers, say $m-1, m$.

\medskip
\noindent
\textbf{(**)$'$} Moreover, if there exists an assigned edge of $G_{e}^{*}$ containing a point $q_{+}$ with $f^{*}(q_{+})>f^{*}(v_{j})$, then each such edge is assigned $m-1$ or $m$ as above. In particular, for a point $q_{-}$ with $f^{*}(q_{-})>f^{*}(v_{j})-\epsilon$, if all of the assigned edges each containing a point $y$ with $f^{*}(y)=f^{*}(q_{-})$ are assigned the same integer $m'$ ($=m-1$ or $m$), then we have: 
\vspace{2mm}

\noindent
\textbf{(i)$'$} all of the assigned edges containing a point $y_{+}$ with $f^{*}(y_{+})>f^{*}(q_{-})$ are assigned $m'$, and  

\noindent
\textbf{(ii)$'$} if a point $y^{*}$ with $f^{*}(y^{*})>f^{*}(q_{-})$ is contained in an assigned edge, then there is a point $y'$ with $f^{*}(y')=f^{*}(q_{-})$ and a path $\mathcal{Q}$ in $G_{e}^{*}$ such that $\mathcal{Q}$ is contained in a union of edges assigned $m'$, and joins the points $y'$ and $y^{*}$.\label{claim} 
\end{lemma}

\begin{proof} 
Let $v_{i}$ be as above. We say that $v_{i}$ is of type~(a) if at the stage just before applying the last Step~2, there does not exist an assigned edge of $G^{*}_{e}$ which contains a point $p$ such that $f^{*}(p)>f^{*}(v_{i})$. The vertex $v_{i}$ is said to be of type~(b) if it is not of type~(a). We say that $v_{j}$ is of type~(a) if at the current stage, there does not exist an assigned edge of $G^{*}_{e}$ which contains a point $q$ such that $f^{*}(q)>f^{*}(v_{j})$. The vertex $v_{j}$ is said to be of type~(b) if it is not of type~(a).

We divide the proof into the following cases. 

\vspace{4mm}

\noindent
\textbf{Case 1.} The vertex $v_{i}$ is of type~(a). 

In this case, we note that $f^{*-1}(f^{*}(v_{i}))\cap$(the union of the assigned edges)$=v_{i}$, and that the valency of $v_{i}$ is three. 

\medskip
\noindent
\textbf{Case 1-1.} The vertex $v_{i}$ satisfies the above condition (*)-(1). 

\begin{figure}[htbp]
 \begin{center}
  \includegraphics[width=90mm]{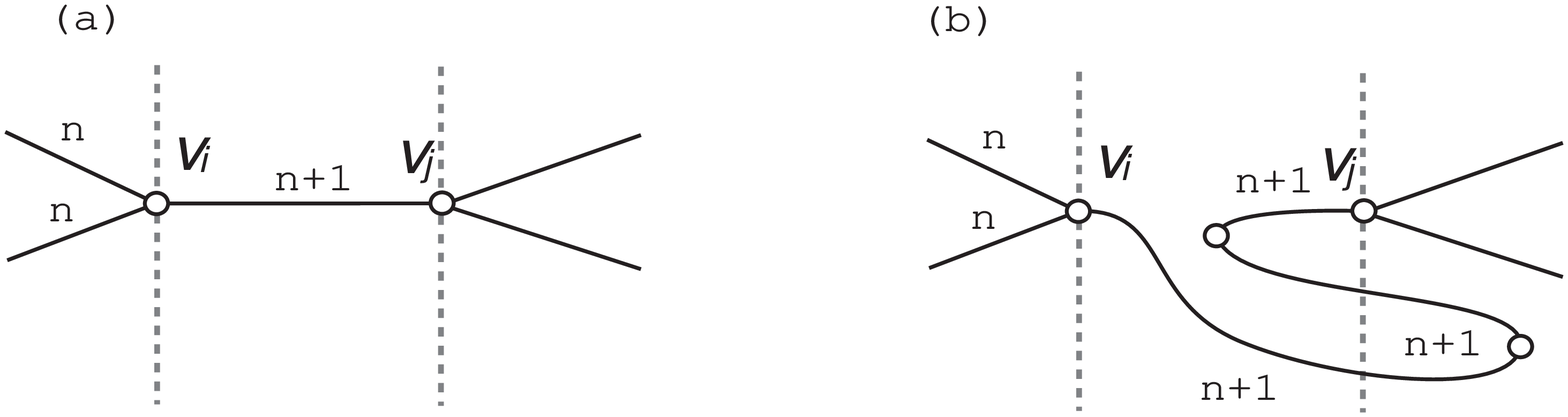}
 \end{center}
 \caption{}
 \label{}
\end{figure}

\vspace{3mm}

In this case, it is clear that we have (*)$'$-(1)$'$ with $m=n+1$. Suppose $v_{j}$ is of type~(a), then it is clear that (**)$'$ holds. Suppose that $v_{j}$ is of type~(b). 
Since the last assigning process is Step 1, all of the assigned edges containing a point $q$ with $f^{*}(q)>f^{*}(v_{i})$ are assigned $n+1$. This implies (**)$'$ holds. We note that in this case, the all of the integers introduces in the process is $n+1$, and there is no possibility of multiple assignment.

\medskip
\noindent
\textbf{Case 1-2.} The vertex $v_{i}$ satisfies the above condition (*)-(2). 

\vspace{-5mm}

\begin{figure}[htbp]
 \begin{center}
  \includegraphics[width=90mm]{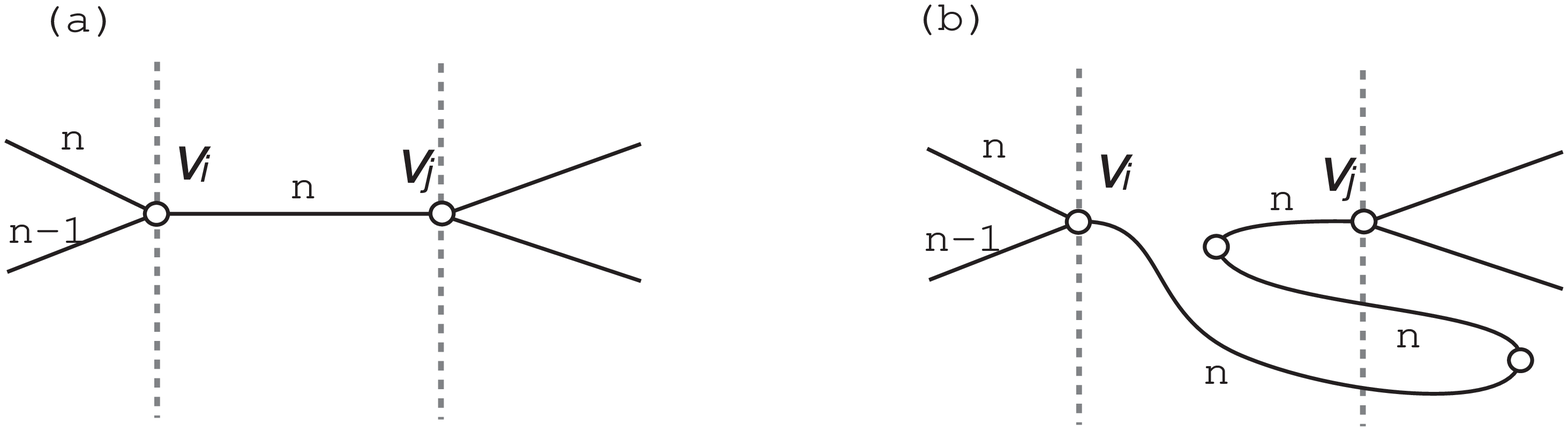}
 \end{center}
 \caption{}
 \label{}
\end{figure}

\vspace{-3.5mm}

In this case, we see that two edges bind into one edge at $v_{i}$, where an edge is assigned $n$ and the other is assigned $n-1$. 
We note that the same arguments as in Case 1-1 works, where the new integer is $n$. 

\medskip

\noindent
\textbf{Case 2.} The vertex $v_{i}$ is of type~(b).

\medskip
\noindent
\textbf{Case 2-1.} The vertex $v_{i}$ satisfies the above condition (*)-(1). 

In this case, the set of the integers assigned to the edges of $G^{*}_{e}$ each of which contains a point $r$ with $f^{*}(r)=f^{*}(v_{i})+\epsilon$ consists of \{$n$, $n+1$\}. (Note that each edge is assigned at most one integer.) Since the assignment at the stage before applying the last Step~2 satisfies (**)-(i), the current assignment satisfies (*)$'$-(1)$'$ or (*)$'$-(2)$'$. Now we will see (**)$'$ holds. 

\medskip

\noindent
\textbf{Case 2-1-1.} The vertex $v_{j}$ satisfies (*)$'$-(1)$'$ with $m=n+1$. 

\vspace{-3mm}

\begin{figure}[htbp]
 \begin{center}
  \includegraphics[width=90mm]{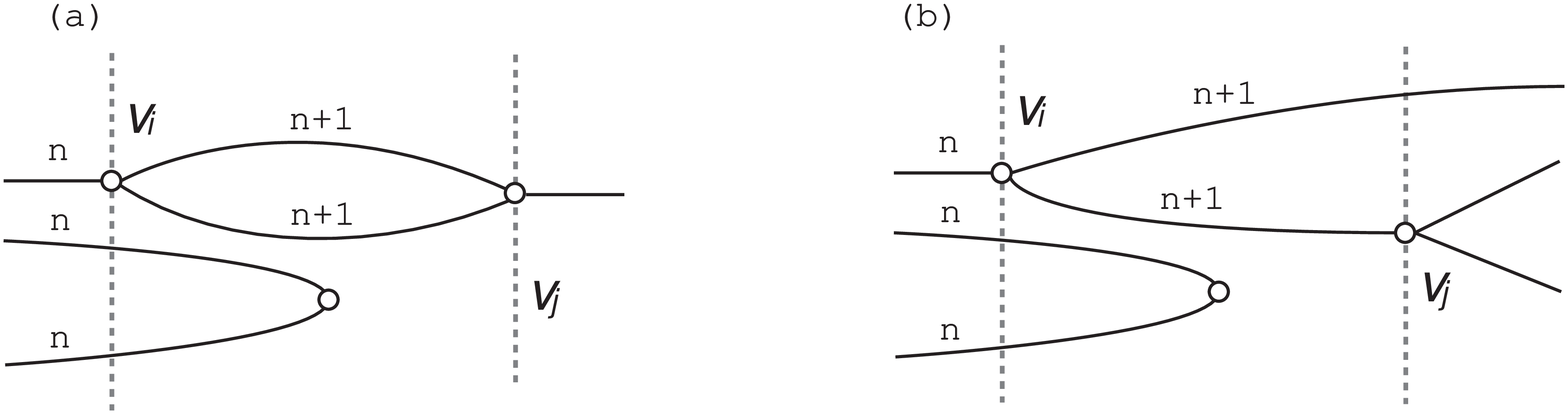}
 \end{center}
 \caption{}
 \label{}
\end{figure}

In this case, we note that at the stage before applying the last Step~2, all of the assigned edges containing a point $r'$ with $f^{*}(r')>f^{*}(v_{i})$ are assigned $n$, and the integer introduced in the last Step~2 and Step~1 is $n+1$. Suppose $v_{j}$ is of type~(a), then it is clear that (**)$'$ holds (see Fig. 18(a)). Suppose that $v_{j}$ is of type~(b) (see Fig. 18(b)). Since $v_{i}$ satisfies (**)-(ii), (**)$'$-(i)$'$ holds with $m'=n+1$. Since the last assigning process is Step~1, we see that (**)$'$-(ii)$'$ holds.

\medskip
\noindent
\textbf{Case 2-1-2.} The vertex $v_{j}$ satisfies (*)$'$-(1)$'$ with $m=n$. 

\begin{figure}[htbp]
 \begin{center}
  \includegraphics[width=90mm]{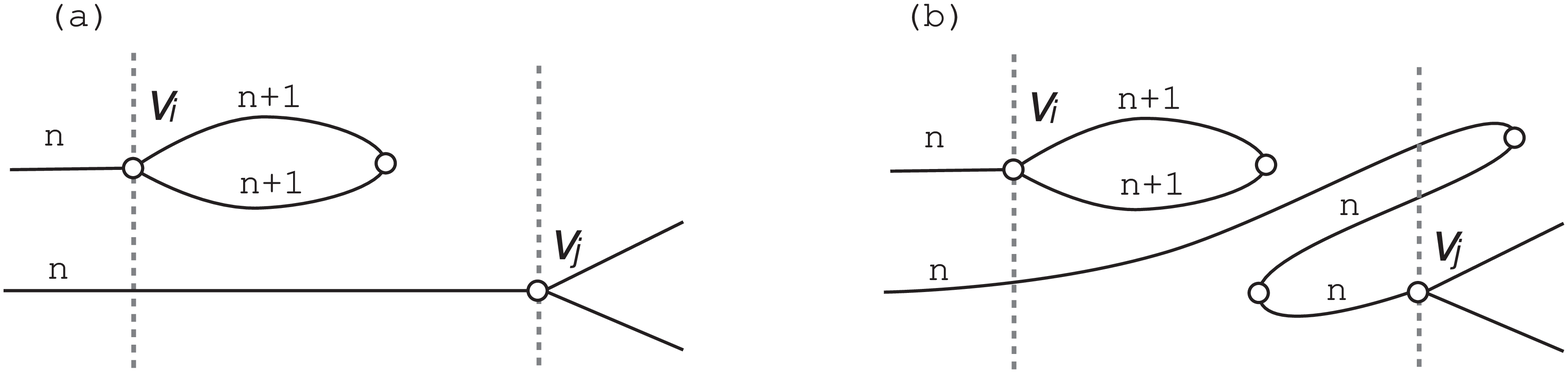}
 \end{center}
 \caption{}
 \label{}
\end{figure}

In this case, we note that at the stage before applying the last Step~2, all of the assigned edges containing a point $r'$ with $f^{*}(r')>f^{*}(v_{i})$ are assigned $n$, and the integer introduced in the last Step~2 and Step~1 is $n+1$. Suppose $v_{j}$ is of type~(a), then it is clear that (**)$'$ holds (see Fig. 19(a)). Suppose that $v_{j}$ is of type~(b). The condition of this case (Case~2-1-2) implies that the union of the edges assigned $n+1$ in these steps is a closed curve lying in the level less than $f^{*}(v_{j})$ (see Fig. 19(b)). Hence by (**)-(i), we see that $v_{j}$ satisfies (**)$'$-(i)$'$, and by (**)-(ii), we see that $v_{j}$ satisfies (**)$'$-(ii)$'$. 

\medskip
\noindent
\textbf{Case 2-1-3.} The vertex $v_{j}$ satisfies (*)$'$-(2)$'$ with $m=n+1$. 

\vspace{3mm}

\begin{figure}[htbp]
 \begin{center}
  \includegraphics[width=90mm]{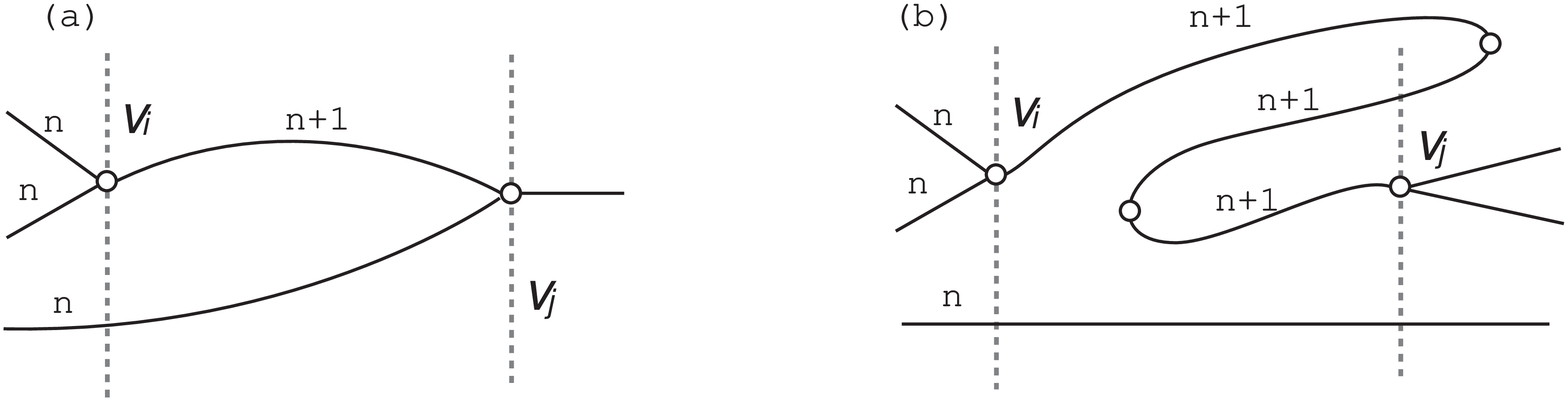}
 \end{center}
 \caption{}
 \label{}
\end{figure}

In this case, we note that at the stage before applying the last Step~2, all of the assigned edges containing a point $r'$ with $f^{*}(r')>f^{*}(v_{i})$ are assigned $n$, and the integer introduced in the last Step~2 and Step~1 is $n+1$. Suppose $v_{j}$ is of type~(a), then it is clear that (**)$'$ holds. Suppose that $v_{j}$ is of type~(b). By (**) and the fact that the integer introduced in the last Step~2 and Step~1 is $n+1$, we see that $v_{j}$ satisfies (**)$'$. 

\medskip

\noindent
\textbf{Case 2-2.} The vertex $v_{i}$ satisfies the above condition (*)-(2). 

In this case, the set of the integers assigned to each edge of $G^{*}_{e}$ containing a point $r$ with $f^{*}(r)=f^{*}(v_{i})+\epsilon$ is either \{$n$\} or \{$n-1$, $n$\}. (Note that each edge is assigned at most one integer.)

\medskip
\noindent
\textbf{Case 2-2-1.} The set of the integers assigned to each edge of $G^{*}_{e}$ containing a point $r$ with $f^{*}(r)=f^{*}(v_{i})+\epsilon$ is \{$n$\}.

\vspace{3mm}

\begin{figure}[htbp]
 \begin{center}
  \includegraphics[width=90mm]{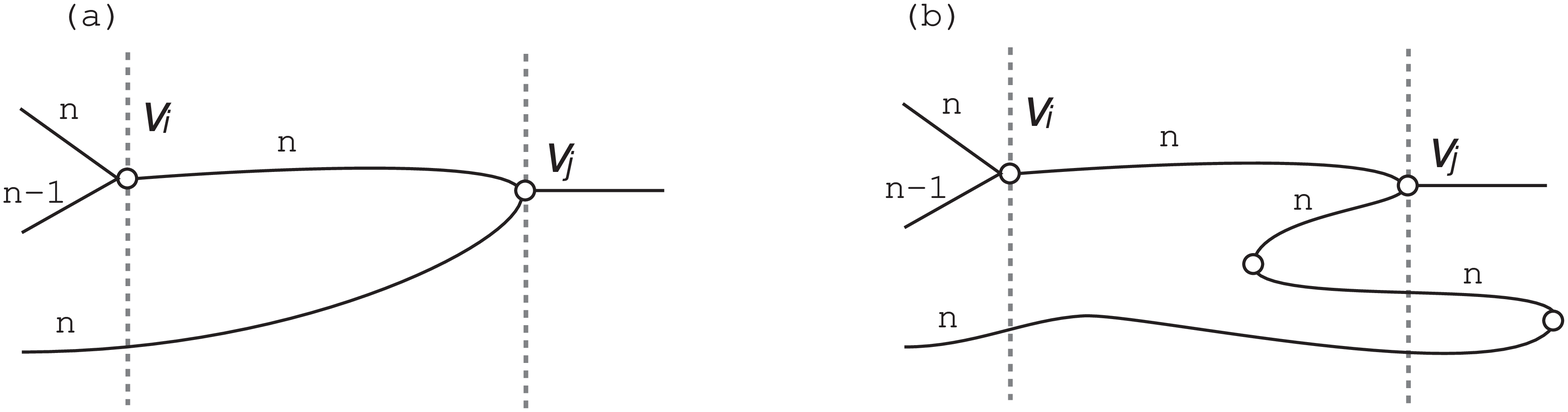}
 \end{center}
 \caption{}
 \label{}
\end{figure}

In this case, by (**), we see that in the current assignment, each edge of $G^{*}_{e}$ containing a point $r'$ with $f^{*}(r')>f^{*}(v_{i})$ is assigned $n$. Suppose $v_{j}$ is of type~(a), then it is clear that (**)$'$ holds (see Fig. 21(a)). Suppose that $v_{j}$ is of type~(b) (see Fig. 21(b)). By (**) and the fact that the integer introduced in the last Step~2, and Step~1 is $n$, we see that $v_{j}$ satisfies (**)$'$. 

\medskip
\noindent
\textbf{Case 2-2-2.} The set of the integers assigned to each edge of $G^{*}_{e}$ containing a point $r$ with $f^{*}(r)=f^{*}(v_{i})+\epsilon$ is \{$n-1, n$\}.

In this case, by using the argument as in Case~2-1, we see that $v_{j}$ satisfies (*)$'$-(1)$'$ or (*)$'$-(2)$'$, and (**)$'$. 
This completes the proof of the lemma. 
\end{proof}

By Lemma \ref{claim}, we can apply Step~2 and Step~1 again and repeat the procedure until all the edges of $G^{*}_{e}$ assigned. 

\begin{figure}[htbp]
 \begin{center}
  \includegraphics[width=58mm]{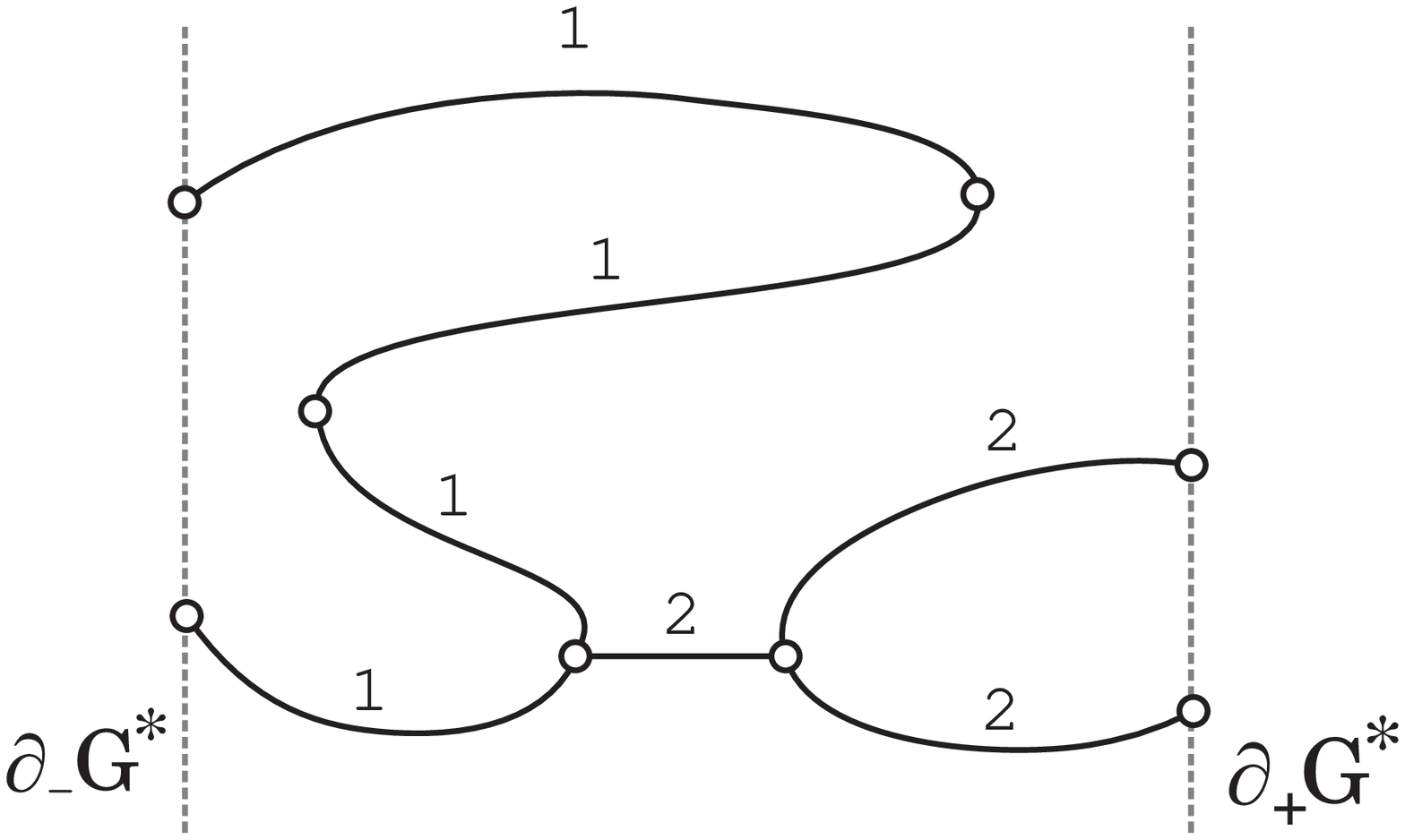}
 \end{center}
 \caption{}
 \label{fig:three}
\end{figure}

Recall that $v_{1},\dots, v_{k}$ are the vertices of $G^{*}_{e}$ which are not $\partial$-vertices such that $f^{*}(v_{1}) < f^{*}(v_{2}) < \dots <f^{*}(v_{k})$. 
Let $r_{0} = s_{0}- \epsilon$ and $r_{k+2} = s_{1} + \epsilon$. Then we fix levels $r_{1},\dots, r_{k+1}$ such that $s_{0} < r_{1} < f^{*}(v_{1})$, $f^{*}(v_{i}) < r_{i+1} < f^{*}(v_{i+1})$ $(i = 1,\dots, k-1)$, and $f^{*}(v_{k}) < r_{k+1} < s_{1}$. We consider the system of simple closed curves in $P_{r_{i}} \cap Q_{t}$ on $P_{r_{i}}$. Since there is a continuous deformation from $P_{r_{i}}$ to $P_{r_{i+1}}$, that is the sweep-out, we may regard $P_{r_{1}} \cap Q_{t}$,\dots, $P_{r_{k+1}} \cap Q_{t}$ are systems of simple closed curves on $P$. Recall that $(r_{0},t)$ is contained in a region labelled $A$. Let $c_{0}$ be the simple closed curve in $P_{r_{0}}\cap Q_{t}$ which bounds a meridian disk in $A$, and let $c_{*}$ be an essential simple closed curve in $P$ corresponding to an edge $e_{*}$ of $G^{*}_{e}$ (that is, $c_{*}$ is a component of $P_{r_{i}}\cap{Q_{t}}$ for some $i$).

\begin{lemma}
$d_{P}(c_{0},c_{*})$ is at most the integer assigned to $e_{*}$.
\end{lemma} 
 
\begin{proof}
We prove the lemma by following the assigning process for the edges of $G^{*}_{e}$ in this section. 

We consider the intersection $P_{r_{i}} \cap Q_{t}$. Let $C_{i}$ be the union of the components of $P_{r_{i}} \cap Q_{t}$ which are essential on $P_{r_{i}}$. Since $(s_{0},s_{1})\times\{t\}$ is an unlabelled interval, each component of $C_{i}$ is essential on $Q_{t}$. 
The isotopy class of $C_{i+1}$ in $P_{r_{i+1}}$ is obtained from the isotopy class of $C_{i}$ in $P_{r_{i}}$ by a band move, or addition or deletion of a pair of simple closed curves which are parallel on $Q_{t}$. 
Hence we see that the distance between $c_{0}$ and each element of $C_{1}$ is at most 1. 
This observation represents the assignment of integer 1 in Step 0.

Then we consider about Step 1. Suppose that there is a valency two vertex $v_{i}$ adjacent to edges $e_{l}, e_{l'}$ such that $e_{l}$ has already been assigned and $e_{l'}$ has not been assigned yet. Let $c_{l}$ (resp. $c_{l'}$) be the simple closed curve corresponding to $e_{l}$ (resp. $e_{l'}$). Note that $c_{l}, c_{l'}$ are pairwise parallel essential simple closed curves on $Q_{t}$. Hence $c_{l}\cup c_{l'}$ bounds an annulus, say $\mathcal{A}$ in $Q_{t}$. 

\medskip

\noindent
\textbf{Claim.} The simple closed curves $c_{l}, c_{l'}$ are isotopic on $P$. 

\vspace{2mm}
\noindent
\textit{Proof.} Since $[s_{0}+\epsilon,s_{1}-\epsilon]\times\{t\}$ goes though unlabelled regions, each component of int$\mathcal{A}\cap{P_{s_{0}+\epsilon}}$ (resp. int$\mathcal{A}\cap{P_{s_{1}-\epsilon}}$) is a simple closed curve that is inessential in both $\mathcal{A}$ and $P_{s_{0}+\epsilon}$ (resp. $\mathcal{A}$ and $P_{s_{1}-\epsilon}$). By using innermost disk arguments if necessary, we may suppose $\mathcal{A}$ is completely contained in $P\times[s_{0}+\epsilon,s_{1}-\epsilon]$. This annulus gives a free homotopy from $c_{l}$ to $c_{l'}$ on $P$. This show that $c_{l}$ and $c_{l'}$ are isotopic on $P$.

\vspace{4mm}
\noindent
The above claim shows that the assigning rule in Step 1 is natural. 

Now we consider about Step 2. Let $c_{r}$ be the simple closed curve corresponding to the point $r$ with $f^{*}(r)=f^{*}(v_{i})-\epsilon$ and $c_{l}$ be the simple closed curve corresponding to the point $l$ with $f^{*}(l)=f^{*}(v_{i})+\epsilon$. Recall that the isotopy class of $C_{i+1}$ is obtained from the isotopy class of $C_{i}$ by a band move, or addition or deletion of a pair of simple closed curves which are parallel on $Q_{t}$. Hence $c_{l}$ is ambient isotopic to simple closed curves disjoint from $c_{r}$. 

Suppose that $v_{i}$ satisfies the condition~(*)-(1) in this section. Since an edge containing the point $x$ is assigned $n$, $d_{P}(c_{0},c_{x})\leq{n}$. Note that $c_{l}$ is disjoint from $c_{r}$. Hence $d_{P}(c_{0},c_{l})\leq{n+1}$. Recall that $n+1$ is the number assigned to $e_{l}$. 

Suppose that $v_{i}$ satisfies the condition~(*)-(2) in this section. We may suppose that the edge containing $r$ is assigned $n-1$, hence $d_{P}(c_{0},c_{r})\leq{n-1}$. Note that $c_{l}$ is disjoint from $c_{r}$. Hence $d_{P}(c_{0},c_{l})\leq{n}$. Recall that $n$ is the number assigned to $e_{l}$. 
This completes the proof of the lemma. 
\end{proof}

Now, we give the proof of the next theorem. 

\begin{theorem}
Let $P, Q$ and $G^{*}_{e}$ be as above. Let $n$ be the minimum of the integers assigned to the edges adjacent to $\partial_{+}G^{*}_{e}$. 
Then the distance $d(P)$ is at most $n + 1$.
\label{main}
\end{theorem}

\begin{proof}
Let $r_{0},\dots, r_{k+2}$ be as above. Recall that $P_{r_{0}} \cap Q_{t}$ contains a simple closed curve, say $c_{0}$ which bounds a meridian disk of $A$. On the other hand, $P_{r_{k+2}} \cap Q_{t}$ contains a simple closed curve, say $c_{k+2}$ which bounds a meridian disk of $B$. 
By Lemma 7.2, we see that $d(c_{0},c')\leq n$, for a component $c'$ of $P_{r_{k+1}}\cap{Q_{t}}$. Note that $c'$ is disjoint from $c_{k+2}$. Hence $d_{P}(c_{0},c_{k+2})\leq n+1$. This implies that $d(P)\leq n+1$. \\
\end{proof}

\section*{Acknowledgement}
I would like to thank Professor Tsuyoshi Kobayashi for many helpful advices and comments. I also thank the referee for careful reading of the first version of the paper.


\end{document}